\documentclass[10pt,twoside,reqno,final]{manuscript}

\usepackage[numbers,sort&compress]{natbib}
\usepackage{multirow}
\usepackage{diagbox} 
\usepackage[notcite,notref]{showkeys}

\hypersetup{
  colorlinks=true,
  linkcolor=magenta,
  citecolor=blue,
  filecolor=blue,
  urlcolor=blue
}

\graphicspath{{./figures/}}

\title{Discovering Dynamics with Kolmogorov–Arnold Networks: Linear Multistep Method-Based Algorithms and Error Estimation} 

\shorttitle{DISCOVERING DYNAMICS WITH KAN}

\author[1]{Jintao Hu}
\author[1]{Hongjiong Tian}
\author[1,\cormark]{Qian Guo}
\affil[1]{Department of Mathematics, Shanghai Normal University, P. R. China}
\cortext{Corresponding author. Email address: \texttt{qguo@math.shnu.edu.cn}}

\shortauthor{J. Hu, H. Tian and Q. Guo}

\date{}

\begin{document}

\maketitle

\begin{abstract}
Uncovering the underlying dynamics from observed data is a critical task in various scientific fields. Recent advances have shown that combining deep learning techniques with linear multistep methods (LMMs) can be highly effective for this purpose. In this work, we propose a novel framework that integrates Kolmogorov–Arnold Networks (KANs) with LMMs for the discovery and approximation of dynamical systems' vector fields. Specifically, we begin by establishing precise error bounds for two-layer B-spline KANs when approximating the governing functions of dynamical systems. Leveraging the approximation capabilities of KANs, we demonstrate that for certain families of LMMs, the total error is constrained within a specific range that accounts for both the method’s step size and the network's approximation accuracy. Additionally, we analyze the difference between the numerical solution obtained from solving the ordinary differential equations with the fitted vector fields and the true solution of the dynamical system. To validate our theoretical results, we provide several numerical examples that highlight the effectiveness of our approach.

\end{abstract}

\section{Introduction}

Dynamical systems are essential for understanding and predicting the evolution of physical phenomena. They encompass a spectrum of mathematical descriptions, from the simplicity of harmonic motion to the intricacies of fluid dynamics and turbulence, providing analytical tools crucial for assessing the historical dependence of system states and forecasting their future trajectories. Typically articulated through exact differential equations rooted in fundamental physical laws—such as conservation of energy, mass, and momentum-traditional first-principles approaches face limitations when applied to highly complex systems. In an era abundant with data and witnessing rapid strides in machine learning, a formidable challenge arises: the possibility of extracting detailed mathematical models directly from data for dynamical systems. This problem holds significant implications not only for theoretical physics and applied mathematics but also poses new challenges for engineering practice and the field of data science \cite{Agarwal2000,BenettinGiorgilli1994,E2017,Thie1968}.

This problem can be summarized as follows: given certain discrete values of the state equation, how can we automatically discover complex dynamical systems from this data? Approaches to system identification generally fall into three major categories: model-driven methods, data-driven methods, and hybrid methods combining both. Data-driven methods focus more on capturing complex relationships, whereas model-driven methods prioritize the interpretability of the system structure. Hybrid methods aim to strike a balance between accuracy and physical consistency. In most real-world scenarios, methods typically involve a combination of these approaches. Representative methods include Gaussian processes, symbolic regression, sparse regression, statistical learning, and linear multistep methods (LMMs) using neural networks.

Gaussian processes \cite{KocijanGirardBankoMurray-Smith2005, RaissiPerdikarisKarniadakis2017, RaissiPerdikarisKarniadakis2017a} adeptly model system states or derivative output relationships using Gaussian distributions, optimizing hyperparameters by maximizing marginal log-likelihood to facilitate system parameter inference. While suitable for low to moderate-dimensional problems, they often require sparse approximations in high-dimensional settings and certain restrictions are imposed on the system form, typically used for parameter estimation in simpler systems \cite{ZhangLin2021}. Symbolic regression \cite{BongardLipson2007,SchmidtLipson2009} generates and refines symbolic models to fit observed data, providing physically meaningful functional forms for equation modeling \cite{RudyNathanKutzBrunton2019}. However, for large or complex systems, it can be computationally expensive and prone to overfitting, often necessitating additional regularization techniques to ensure generalization capability. Sparse regression \cite{BruntonProctorKutz2016,RudyBruntonProctorKutz2017,ShengGuang2018} approximates system control or state functions by identifying a sparse combination from a candidate basis function set, with coefficients determined by sparse regression algorithms. This approach offers an explicit system formula with minimal prior knowledge requirements \cite{ZhangLin2021}, but its efficiency may be compromised or it may fail to accurately fit complex dynamical models if the system lacks a simple or sparse representation \cite{LuZhongTangMaggioni2019,LongLuDong2019,RudyNathanKutzBrunton2019}. Statistical learning methods \cite{LuZhongTangMaggioni2019,ZhongMillerMaggioni2020} minimize empirical error to select an appropriate kernel function within a hypothesis space, learning the system's interaction relationships. This approach mitigates the curse of dimensionality to some extent and can be used to identify interaction patterns in high-dimensional systems \cite{LuZhongTangMaggioni2019}, but it is limited to dynamical systems that can be approximately represented through kernel functions \cite{FengKulickRenTang2024}. LMMs with neural networks \cite{RaissiPerdikarisKarniadakis2018,TipireddyPerdikarisStinisTartakovsky2019,XieZhangWebster2019} approximate control functions via neural networks, identifying control functions by minimizing the residuals of the discretized dynamical system through LMMs. This approach can achieve high convergence orders and handle complex or high-dimensional systems, as demonstrated in \cite{LongLuDong2019,RudyNathanKutzBrunton2019,ZhangLin2021}. The stability and convergence of dynamic system identification have been thoroughly analyzed in \cite{DuGuYangZhou,DDLMM2021}, where the relationship between the overall error, the error of the approximating network, and the residual of the LMM-discretized dynamic system was established.

In light of these considerations, this study delves into the convergence theory of the LMMs framework integrated with deep learning, employing Kolmogorov–Arnold Networks (KANs) as the neural network model. Through meticulous error analysis of KANs, we derive the discrepancy between the numerical solution $\boldsymbol{x}_{\mathcal{NN}}$ obtained from the learned network and the target state function $\boldsymbol{x}$.

The structure of the paper is outlined as follows: Sections 2.1 and 2.2 provide an introduction to the background of Kolmogorov–Arnold Networks (KANs) and the construction of the proposed network. Sections 2.3 and 2.4 detail the upper and lower error bounds for the proposed two-layer B-spline KANs, respectively. Section 3 offers an overview of dynamical systems and the foundational knowledge of Linear Multistep Methods (LMMs), followed by a description of the application of LMMs in dynamic identification. Section 4 presents a comprehensive error analysis of the LMMs when applied in conjunction with the two-layer B-spline KANs. Numerical experiments to validate the theoretical error results are conducted in Section 5.

\section{B-Spline KANs Approximation}
The theoretical foundation of traditional Multi-layer Perceptrons (MLPs) is often ascribed to the universal approximation theorem. This study, however, pivots attention towards the Kolmogorov-Arnold representation theorem, which serves as the theoretical basis for a neural network known as KANs. Through an in-depth analysis of the network's constituent basis functions, this research delineates a rigorous quantification of the approximation error within a two-layer KANs framework employing B-spline functions. The approximation is further grounded in Linear Multistep Methods (LMMs) for trajectory-based dynamical system identification, incorporating implicit regularization.  Moreover, this study employs the concept of Vapnik-Chervonenkis (VC) dimension from computational learning theory to ascertain the minimum error bound for networks that leverage B-spline basis functions.

\subsection{Kolmogorov-Arnold theorem}

Braun and Griebel \cite{BraunGriebel2009} have elucidated that any continuous multivariate function $f$ defined on a bounded domain can be expressed as a finite sum of continuous univariate functions, with addition as the combining operation. In essence, this representation technique is applicable to continuous functions within the specified domain.

\begin{lemma}[Kolmogorov-Arnold Representation theorem \cite{BraunGriebel2009}]\label{lem:KAT}
Let $f: [0,1]^{d} \rightarrow \mathbb{R}$ be an arbitrary multivariate continuous function. Then it has the representation
\begin{equation}\label{eq:KAT}
f(\boldsymbol{x})=\sum_{q=0}^{2d}\phi_q\biggl(\sum_{p=1}^d \psi_{q,p}(x_p)\biggr) \quad \text{for any} ~ \boldsymbol{x}=(x_1,\ldots,x_d)\in[0,1]^d
\end{equation}
with continuous one-dimensional outer and inner functions $\phi_q$ and $\psi_{q, p}$. All these functions $\phi_q$, $\psi_{q, p}$ are defined on the real line. The inner functions $\psi_{q, p}$ are independent of the function $f$.
\end{lemma}

The foundational role of addition is emphasized by the demonstration that any multivariate function can be represented as a combination of univariate functions through addition. This finding may be seen as encouraging news for machine learning, suggesting that the challenge of learning high-dimensional functions could be simplified to learning a polynomial number of one-dimensional functions, potentially easing the complexity of the task. However, the situation is complicated by the fact that these one-dimensional functions may exhibit non-smooth or fractal characteristics, which pose significant challenges for their learnability in practice \cite{Kurkova1991,PoggioBanburskiLiao2020}, particularly for optimization algorithms. As a result, the Kolmogorov-Arnold representation theorem has largely been overlooked in the field of machine learning, viewed as theoretically sound yet often lacking practical applicability \cite{Kurkova1991,PoggioBanburskiLiao2020}.

However, we believe in the significant potential of the Kolmogorov-Arnold theorem for machine learning. First of all, we are not limited to the original formulation \eqref{eq:KAT}, which features a two-layer structure with nonlinearities and a limited number of terms, specifically $(2d + 1)$, in the hidden layer.
We simplify the multi-layer KANs proposed in \cite{LiuWangVaidyaRuehleEtAl2024} by considering a two-layer KANs where the first layer has a width of dimension $d$ and the second layer has a width of $N$. This allows us to obtain a precise estimation of the error for our network. The specific construction of the network will be detailed in the next subsection.

\subsection{B-Spline KANs framework}
Consider a scenario within the realm of supervised learning where the objective is to approximate a function \( f \) such that for a given set of input-output pairs \(\{\boldsymbol{x}_i, y_i\}\), the relationship \( y_{i} \approx f(\boldsymbol{x}_i) \) holds for every data point. \autoref{lem:KAT} suggests that this objective can be achieved by identifying suitable univariate functions \(\psi_{q, p}\) and \(\phi_q\). Kolmogorov-Arnold Networks (KANs), a class of neural networks predicated on the Kolmogorov-Arnold representation theorem, offer a novel approach to this end. According to the theorem, any continuous multivariate function \( f(\boldsymbol{x}) = f(x_1, x_2, \ldots) \) defined over a finite domain can be decomposed into a finite combination of continuous univariate functions, with addition serving as the binary operation that combines these elements.

The approximation for $f(\boldsymbol{x})$ was proposed in \cite{LiuWangVaidyaRuehleEtAl2024} as follows:
\begin{equation}\label{eq:KANa}
f(\boldsymbol{x}) = \sum_{i_{L-1}=1}^{n_{L-1}} \phi_{L-1,i_{L},i_{L-1}} \left( \sum_{i_{L-2}=1}^{n_{L-2}} \cdots \left( \sum_{i_1=1}^{n_1}\phi_{1,i_1,i_0} \left( \sum_{i_0=1}^{n_0} \phi_{0,i_1,i_0}(x_{i_0}) \right) \right)  \cdots \right).
\end{equation}
The deep network in question is designed as a generalization of the Kolmogorov-Arnold representation theorem. Below, we will explain the formulas and concepts that appear in the equations, simplify them to obtain our two-layer B-spline KANs, and discuss the details of our networks implementation.

The main notations are listed as follows.
\begin{itemize}
  \item Vectors and matrices are denoted in a bold font. Standard vectorization is adopted in the matrix and vector computation. For example, adding a scalar and a vector means adding the scalar to each entry of the vector.
  \item KAN layer with $d_{in}$-dimensional inputs and $d_{out}$-dimensional outputs can be defined as a matrix of 1D functions
      \begin{equation}\label{2.1}
    \boldsymbol{\Phi}=\left\{\phi_{q,p}\right\},~~p=1,2,\ldots,d_{in}, ~ q=1,2, \ldots,d_{out},
      \end{equation}
      where the functions $\phi_{q,p}$ have trainable parameters, as detaild below. In the Kolmogov-Arnold theorem, the inner functions form a KAN layer with $d_{in} = d$ and $d_{out} = 2d + 1$, and the outer functions form a KAN layer with $d_{in} = 2d + 1$ and $d_{out}=1$. So the Kolmogorov-Arnold representations in \eqref{eq:KAT} are simply compositions of two-layer KANs .
  \item The shape of a KAN is represented by an integer array
  $$\left[d_{0},d_{1},\ldots,d_{L}\right],$$
  where $d_{i}$ is the number of nodes in the $i^{th}$ layer of the computational graph. We denote the $i^{th}$  neuron in the $l^{th}$  layer by $(l,i)$, and the activation value of the $(l,i)$ -neuron by $x_{l,i}$. Between layer $l$ and layer $l+1$, there are $n_{l}n_{l+1}$ activation functions: the activation function that connects $(l,i)$ and $(l+1,j)$ is denoted by
  \begin{equation}\label{eq:act_fun}
    \phi_{l,j,i},~~ l=0,\ldots,L-1,~ i=1,\ldots,n_{l},~ j=1,\ldots, n_{l+1}.
  \end{equation}

  \item The pre-activation of $\phi_{l,j,i}$ is simply $x_{l,i}$, the post-activation of $\phi_{l,j,i}(x_{l,i})$ is denoted by $ \tilde{x}_{l,j,i} \equiv \phi_{l,j,i}(x_{l,i}) $. The activation value of the $(l+1,j)$ neuron is simply the sum of all incoming postactivations:
  \begin{equation*}
     x_{l+1,j}=\sum_{i=1}^{n_{l}}\tilde{x}_{l,j,i}=\sum_{i=1} ^{n_{l}}\phi_{l,j,i}(x_{l,i}), ~~j=1,\cdots,n_{l+1}.
  \end{equation*}
  In matrix form, this reads
    \begin{equation}
    \boldsymbol{x}_{l+1} = \underbrace{\left(\begin{array}{cccc}
    \phi_{l, 1,1}(\cdot) & \phi_{l,1,2}(\cdot) & \cdots & \phi_{l,1, n_l}(\cdot)\\
    \phi_{l, 2,1}(\cdot) & \phi_{l, 2, 2}(\cdot) & \cdots & \phi_{l, 2, n_l}(\cdot)\\
    \vdots & \vdots & & \vdots\\
    \phi_{l, n_{l+1}, 1}(\cdot) & \phi_{l, n_{l+1}, 2}(\cdot) & \cdots & \phi_{l, n_{l+1}, n_l}(\cdot)
    \end{array}\right)}_{\boldsymbol{\Phi_l}} \boldsymbol{x}_l,
    \end{equation}
 where $\boldsymbol{\Phi_l}$ is the function matrix corresponding to the $l^{th}$ KAN layer. A general KAN is a composition of $L$ layers: given an input vector $\boldsymbol{x}_0 \in R^{n_{0}}$, the output of KAN is
 \begin{equation}\label{eq:KAN}
   \operatorname{KAN}(\boldsymbol{x})=(\boldsymbol{\Phi_{L-1}} \circ \boldsymbol{\Phi_{L-2}} \circ \ldots \circ \boldsymbol{\Phi_{1}} \circ \boldsymbol{\Phi_{0}})\boldsymbol{x} .
 \end{equation}

\end{itemize}

For the existing multilayer KAN, the main purpose of this paper is to study the lower bound of the approximation error of any continuous function $f$ over a given finite interval using B-spline functions with a specified number of layers. This is related to the number of grid points
$G$ of the constructed B-spline basis functions, as well as the degree of the B-splines. Below, we will first introduce the structure of the two-layer B-spline KAN we are considering, and in the next section, we will provide an exact expression for our error lower bound and prove it.

\begin{figure}[t]
  \centering
  \includegraphics[width=1\linewidth]{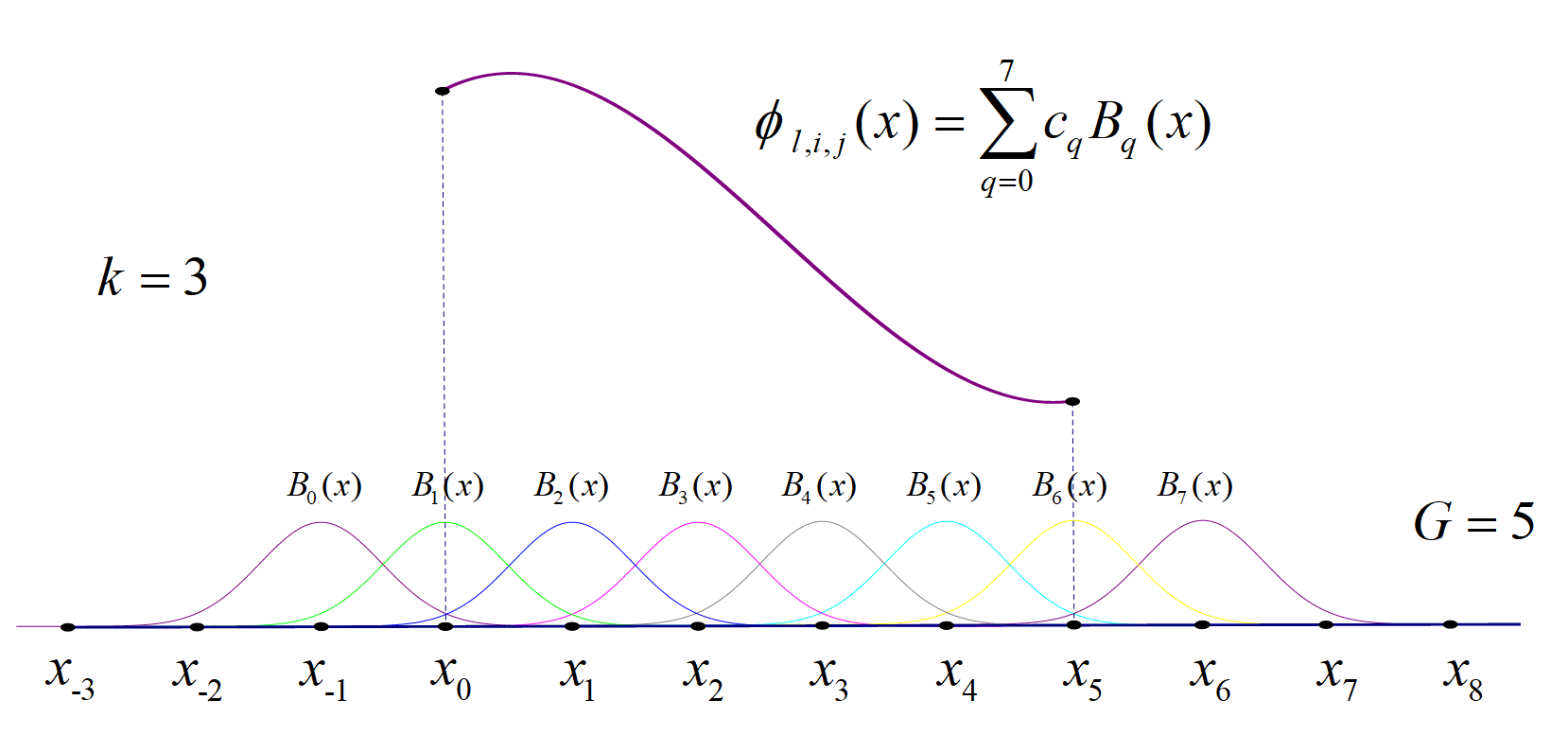}
  \caption{The activation function is parameterized by a set of B-spline basis functions.}
  \label{fig:Bspline1}
\end{figure}

\subsection{B-Spline KANs approximation}
Let's review the two-layer B-spline KANs earlier. Mathematically speaking,

\begin{equation}\label{Th:2layKAT}
\phi(\boldsymbol{x})=\sum_{j=1}^{N}\phi_{1,j} \biggl(\sum_{i=1}^d \phi_{0,j,i}(x_i)\biggr)\quad\mathrm{for~any~} \boldsymbol{x}=(x_1,\ldots,x_d)\in[0,1]^d,
\end{equation}
where $\phi_{0,j,i}$ represents the inner B-spline activation function, collectively denoted as $\Phi_0(\boldsymbol{x})$. $\phi_{1,j}$ denotes the outer B-spline activation function, also collectively denoted as \(\Phi_1(\boldsymbol{x})\), where each B-spline activation function is similarly represented by a set of basis functions, $\phi_{l,i,j}(\boldsymbol{x}) = \sum_{q=0}^{G+k-1} c_q B_q(x)$. $x_{l,i}$ represents the $i$-th element in the $l$-th layer, where $l = 1, 2$ and $i = 1, 2, \ldots, N$. Next, we will derive the approximation properties of the KANs for the above network structure. For convenience, we present all notations used throughout this dissertation.
\begin{definition}[Modulus of continuity]\label{Def:Modcontinuity}
Given $S\subseteq R^{d},$ we can define the modulus of continuity of a continuous function $f$ by
\begin{equation}\label{Def:modulus}
  \omega_f^{S}(r)=\sup\left\{|f(\boldsymbol{x})-f(\boldsymbol{y})|:\|\boldsymbol{x}-\boldsymbol{y}
  \|_2\leq r,\:\boldsymbol{x},\boldsymbol{y}\in S\right\} \quad\text{for any }r\geq0.
\end{equation}
Clearly, $\omega_{f}^{S}(cr)\leq c \omega_{f}^{S}(r)$ for any $c \in \mathbb{N}^{+}$ and $r \geq 0$.
\end{definition}

\begin{definition}[H\"older continuous functions]\label{Def:Holder_continuous}
Given $E\subseteq R^{d},$ a function $f: E \to \mathbb{R} $ is called Hölder continuous if there exist constants $\lambda > 0 $ and $\alpha \in  (0,1] $ such that for all $\boldsymbol{x},\boldsymbol{y} \in E$,
$$|f(\boldsymbol{x}) - f(\boldsymbol{y})| \leq \lambda \|\boldsymbol{x} - \boldsymbol{y}\|^\alpha,$$
which can be denoted as the class $ H_{\lambda}(C^{\alpha}(E)) $. Here, $\| \cdot \|$ represents the distance between $\boldsymbol{x}$ and $\boldsymbol{y}$ (typically the Euclidean distance). In particular, $\alpha=1$, the Hölder continuous function is referred to as a Lipschitz continuous function, meaning there exists a constant $L$ such that:
$$|f(\boldsymbol{x}) - f(\boldsymbol{y})| \leq L \|\boldsymbol{x} - \boldsymbol{y}\|\quad\text{for all}~ \boldsymbol{x}, \boldsymbol{y} \in E.$$
\end{definition}

\begin{lemma}[Extension of Continuous Functions \cite{ShenYangZhang2020}]\label{lem:Extension_of_Continuousf}
Suppose $f$ is a uniformly continuous function defined on a subset $E \subset S$, where $S$ is a metric space with a metric $d_{S}(\cdot , \cdot)$, then there exists a uniformly continuous function $g$ on $S$ such that $f(\boldsymbol{x})=g(\boldsymbol{x})$ for $\boldsymbol{x}\in E$ and $\omega_{f}^{E}(r)=\omega_{g}^{S}(r)$ for any  $r\geq 0$.
\end{lemma}

\begin{lemma}[B-spline function error approximation \cite{deBoor1978}]\label{lem:upperbound_of_Bspline}
  Let $f(x)$ be a continuous functions on $[a,b]$, there exists a function $\phi(x) = \sum_{q=0}^{G+k-1} c_q B_q(x)$ implemented by B-spline basis functions $B_{i,k}(x)$ with degree $k$ of the piecewise B-spline polynomial and number of knots $G$ such that
\begin{equation}\label{eq:upperbound_of_Bspline}
  \|f(x)-\phi(x) \| \leq \omega_{f}\left(\min \left\{\frac{b-a}{\sqrt{2k-2}}, \sqrt{\frac{k}{12}} \cdot \frac{b-a}{G}\right\}\right),
\end{equation}
where $B_{i,k}(x)$  can be uniquely defined using a recurrence relation.
Frist, we have
$$B_{i,0}(x) =
\begin{cases}
1, & x_i \leq x < x_{i+1}, \\
0, & \text{otherwise}.
\end{cases}$$
Then we have
$$B_{i, k}(x)=\frac{x-x_{i}}{x_{i+k}-x_{i}} B_{i, k-1}(x)+\frac{x_{i+k+1}-x}{x_{i+k+1}-x_{i+1}} B_{i+1, k-1}(x)$$
and the basis functions constructed in this way are uniquely defined.
\end{lemma}

\begin{theorem}[Upper bound of two-layer B-spline KANs]\label{thm:upperbound_of_Bspline}
Let $f$ be a continuous functions on $[0,1]^d$, suppose that a function $f(\boldsymbol{x})$ admits a representation
$$ f(\boldsymbol{x})=(\boldsymbol{\Phi}_1\circ\boldsymbol{\Phi}_0)\boldsymbol{x},$$
for any $k,G \in \mathbb{N}^+$, there exists a function $\phi(\boldsymbol{x})=(\boldsymbol{\Phi}_{1}^{G}\circ \boldsymbol{\Phi}_{0}^{G})\boldsymbol{x}$, implemented by a two-layer B-spline KANs with degree $k$ of the piecewise B-spline polynomial and number of knots $G$ such that
\begin{equation}\label{eq:upperbound_of_BsplineKAN}
  \|f(\boldsymbol{x})-\phi(\boldsymbol{x}) \| \leq N(L_{G}d+1) \cdot \omega_{f}\left(\min \left\{\frac{1}{\sqrt{2k-2}}, \sqrt{\frac{k}{12}} \cdot \frac{1}{G}\right\}\right),
\end{equation}
where $d$ denotes the dimension of the input vector, $N$ represents the number of basis functions in the second-layer network, and $L_{G}$ is the Lipschitz constant of $\phi^{G}(\boldsymbol{x})$.
\end{theorem}

\begin{proof}
To establish the inequality in equation \eqref{eq:error}, we start by decomposing it into two parts:
\begin{equation}\label{eq:error}
\begin{aligned}
\left\|(\boldsymbol{\Phi}_{1} \circ \boldsymbol{\Phi}_{0}) \boldsymbol{x}-(\boldsymbol{\Phi}_{1}^{G} \circ \boldsymbol{\Phi}_{0}^{G}) \boldsymbol{x} \right\|  & \leq \left\|(\boldsymbol{\Phi}_{1} \circ \boldsymbol{\Phi}_{0}) \boldsymbol{x}- (\boldsymbol{\Phi}_{1}^{G} \circ \boldsymbol{\Phi}_{0}) \boldsymbol{x} \right\|  \\
& + \left\| (\boldsymbol{\Phi}_{1}^{G} \circ \boldsymbol{\Phi}_{0}) \boldsymbol{x}- (\boldsymbol{\Phi}_{1}^{G} \circ \boldsymbol{\Phi}_{0}^{G}) \boldsymbol{x} \right\|.
\end{aligned}
\end{equation}
We first address the first term. Let
$\boldsymbol{\Phi}_{0}\boldsymbol{x} = \boldsymbol{y}$. By invoking the classical 1D B-spline theory from Lemma \ref{lem:upperbound_of_Bspline}, we obtain:
\begin{equation*}
\begin{aligned}
\left\|(\boldsymbol{\Phi}_{1} \circ \boldsymbol{\Phi}_{0}) \boldsymbol{x}- (\boldsymbol{\Phi}_{1}^{G} \circ \boldsymbol{\Phi}_{0}) \boldsymbol{x} \right\| & = \Bigg\| \sum_{j=1}^{N} \phi_{1, j}(\boldsymbol{y}) - \sum_{j=1}^{N} \phi_{1, j}^{G}(\boldsymbol{y}) \Bigg\|  \\
 & \leq N \cdot \omega_{f}\left(\min \left\{\frac{1}{\sqrt{2k-2}}, \sqrt{\frac{k}{12}} \cdot \frac{1}{G} \right\}\right).
\end{aligned}
\end{equation*}
Next, we consider the second term on the right-hand side of \eqref{eq:error}. Utilizing the Lipschitz constant $L_{G}$ of $\phi^{G}(\boldsymbol{x})$ and $\boldsymbol{\Phi}_{0}^{G}\boldsymbol{x} = \boldsymbol{\bar{y}}$, we have:
\begin{equation*}
\begin{aligned}
\left\| (\boldsymbol{\Phi}_{1}^{G} \circ \boldsymbol{\Phi}_{0}) \boldsymbol{x}- (\boldsymbol{\Phi}_{1}^{G} \circ \boldsymbol{\Phi}_{0}^{G}) \boldsymbol{x} \right\|  & \leq
\Bigg\| \sum_{j=1}^{N} \phi_{1, j}^{G}(\boldsymbol{y}) - \sum_{j=1}^{N} \phi_{1, j}^{G}(\boldsymbol{\bar{y}}) \Bigg\|  \\
& \leq \Bigg\| \sum_{j=1}^{N} (\phi_{1, j}^{G}(\boldsymbol{y}) -  \phi_{1, j}^{G}(\boldsymbol{\bar{y}})) \Bigg\|  \\
& \leq NL_{G} \left\| \boldsymbol{y} -  \boldsymbol{\bar{y}} \right\| = NL_{G} \left\| \boldsymbol{\Phi}_{0}\boldsymbol{x} -\boldsymbol{\Phi}_{0}^{G}\boldsymbol{x} \right\| \\
& \leq  NL_{G}d \cdot \omega_{f}\left(\min \left\{\frac{1}{\sqrt{2k-2}},\sqrt{\frac{k}{12}} \cdot \frac{1}{G}\right\}\right).
\end{aligned}
\end{equation*}
Given the definitions:
$$\Phi_{0,j}(\boldsymbol{x})=\sum_{i=1}^{d} \phi_{0,j, i}(x_i), \quad \Phi_{0,j}^{G}(\boldsymbol{x})=\sum_{i=1}^{d} \phi_{0,j, i}^{G}(x_i),$$
we can express:
$$ \boldsymbol{\Phi}_{0}\boldsymbol{x}  = [\Phi_{0,1}(\boldsymbol{x}),\Phi_{0,2}(\boldsymbol{x}), \ldots , \Phi_{0,d}(\boldsymbol{x})]^{\mathrm{T}},$$
$$ \boldsymbol{\Phi}_{0}^{G}\boldsymbol{x}  = [\Phi_{0,1}^{G}(\boldsymbol{x}),\Phi_{0,2}^{G}(\boldsymbol{x}), \ldots , \Phi_{0,d}^{G}(\boldsymbol{x})]^{\mathrm{T}},$$
$$ \boldsymbol{\Phi}_{0}\boldsymbol{x} -\boldsymbol{\Phi}_{0}^{G}\boldsymbol{x} = [\Phi_{0,1}(\boldsymbol{x}) -\Phi_{0,1}^{G}(\boldsymbol{x}),\Phi_{0,2}(\boldsymbol{x})-\Phi_{0,2}^{G}(\boldsymbol{x}), \ldots , \Phi_{0,d}(\boldsymbol{x})-\Phi_{0,d}^{G}(\boldsymbol{x})]^{\mathrm{T}}, $$
$$ \left\| \boldsymbol{\Phi}_{0}\boldsymbol{x} -\boldsymbol{\Phi}_{0}^{G}\boldsymbol{x} \right\|  \leq d \cdot \omega_{f}\left(\min \left\{\frac{1}{\sqrt{2k-2}}, \sqrt{\frac{k}{12}} \cdot \frac{1}{G} \right\}\right).$$
Returning to \eqref{eq:error}, we combine the results:
\begin{equation*}
\begin{aligned}
\left\|(\boldsymbol{\Phi}_{1} \circ \boldsymbol{\Phi}_{0}) \boldsymbol{x}-(\boldsymbol{\Phi}_{1}^{G} \circ \boldsymbol{\Phi}_{0}^{G}) \boldsymbol{x} \right\|  & \leq N \cdot \omega_{f}\left(\min \left\{\frac{1}{\sqrt{2k-2}}, \sqrt{\frac{k}{12}} \cdot \frac{1}{G} \right\}\right)  \\
& + NL_{G}d \cdot \omega_{f}\left(\min \left\{\frac{1}{\sqrt{2k-2}}, \sqrt{\frac{k}{12}} \cdot \frac{1}{G}\right\}\right) \\
& \leq  N(L_{G}d+1) \cdot \omega_{f}\left(\min \left\{\frac{1}{\sqrt{2k-2}}, \sqrt{\frac{k}{12}} \cdot \frac{1}{G} \right\}\right),
\end{aligned}
\end{equation*}
which satisfies equation \eqref{eq:upperbound_of_BsplineKAN}.
\end{proof}

\begin{corollary}\label{cor:upperbound_of_Bspline_R}
Let $f$ be a continuous functions on $[-R,R]^d$, where $R$ is an arbitrary positive real number, suppose that a function $f(\boldsymbol{x})$ admits a representation
$$ f(\boldsymbol{x})=(\boldsymbol{\Phi}_1\circ\boldsymbol{\Phi}_0)\boldsymbol{x},$$
for any $k,G \in \mathbb{N}^+$, there exists a function $\phi(\boldsymbol{x})=(\boldsymbol{\Phi}_{1}^{G}\circ \boldsymbol{\Phi}_{0}^{G})\boldsymbol{x}$, implemented by a two-layer B-spline KANs with degree $k$ of the piecewise B-spline polynomial and number of knots $G$ such that
\begin{equation}\label{eq:upperbound_of_BsplineKAN_R}
  \left\|f(\boldsymbol{x})-\phi
  (\boldsymbol{x}) \right\| \leq N(L_{G}d+1) \cdot \omega_{f}\left(\min \left\{\frac{2R}{\sqrt{2k-2}}, \sqrt{\frac{k}{12}} \cdot \frac{2R}{G}\right\}\right),
\end{equation}
where $d$ denotes the dimension of the input vector, $N$ represents the number of basis functions in the second-layer network, and $L_{G}$ is the Lipschitz constant of $\phi^{G}(\boldsymbol{x})$.
\end{corollary}
\begin{proof}
Let $f\in C(E)$ be an arbitrary component of $\boldsymbol{f}$ and $f$ is a continuous function. By \autoref{lem:Extension_of_Continuousf} setting $S=\mathbb{R}^{d}$, there exists $h\in C(\mathbb{R}^{d})$ such that $h(\boldsymbol{x})=f(\boldsymbol{x})$ for any $\boldsymbol{x} \in E \subseteq [-R,R]^d$ and $\omega_{h}^{S}( r)=\omega_{f}^{E}(r)$ for any $r \geq 0$.
Define
\begin{equation*}
  \hat{h}(\boldsymbol{x}):= h(2R\boldsymbol{x} -R) \quad \text{for any}~ \boldsymbol{x} \in \mathbb{R}^d.
\end{equation*}
According to \autoref{thm:upperbound_of_Bspline} there exist $\hat{\phi}(\boldsymbol{x})$ such that
\begin{equation*}
  \left\|\hat{h}(\boldsymbol{x})-\hat{\phi}(\boldsymbol{x}) \right\| \leq N(Ld+1) \cdot \omega_{\hat{h}}^{S} \left(\min \left\{\frac{1}{\sqrt{2k-2}}, \sqrt{\frac{k}{12}} \cdot \frac{1}{G}\right\}\right), \quad \text{for any}~ \boldsymbol{x} \in [0,1]^d.
\end{equation*}
By implementing a variable transformation to revert our notation from $\hat{h}$ to $h$, we establish the relationship $f(\boldsymbol{x}) = h(\boldsymbol{x}) = \hat{h}\left(\frac{\boldsymbol{x}+R}{2R}\right)$ for any $\boldsymbol{x} \in E \subseteq [-R,R]^d$. Accordingly, the modulus of $\hat{h}$, denoted as $\omega_{\hat{h}}^{S}(r)$, is equivalent to $\omega_{h}^{S}(2Rr)$, which in turn equals $\omega_{f}^{E}(2Rr)$.

For the function $\phi(\boldsymbol{x})$, we define it in terms of $\hat{\phi}$ as $\phi(\boldsymbol{x}) := \hat{\phi}\left(\frac{\boldsymbol{x}+R}{2R}\right)$ for $\frac{\boldsymbol{x}+R}{2R} \in [0,1]^d$. This leads us to the following inequality:
\begin{equation*}
\begin{aligned}
\|f(\boldsymbol{x})-\phi(\boldsymbol{x})\| & = \|h(\boldsymbol{x})-\phi(\boldsymbol{x})\| = \left\| \hat{h}\left(\frac{\boldsymbol{x}+R}{2R}\right) - \hat{\phi}
\left(\frac{\boldsymbol{x}+R}{2R}\right)\right\|   \\
& \leq N(L_{G}d+1) \cdot \omega_{\hat{h}}^{S} \left(\min \left\{\frac{2R}{\sqrt{2k-2}}, \sqrt{\frac{k}{12}} \cdot \frac{2R}{G}\right\}\right)   \\
& \leq N(L_{G}d+1) \cdot \omega_{f}\left(\min \left\{\frac{2R}{\sqrt{2k-2}}, \sqrt{\frac{k}{12}} \cdot \frac{2R}{G}\right\}\right),
\end{aligned}
\end{equation*}
which implies \eqref{eq:upperbound_of_BsplineKAN_R}.
\end{proof}

\begin{theorem}\label{thm:upperbound_of_Bspline for holder}
Let $f$ be a function that is H\"older continuous with exponent
$\alpha$ and constant $\lambda$  on the domain $[0,1]^d$.
Assume that $f(\boldsymbol{x})$ can be expressed as a composition of two functions:
$$ f(\boldsymbol{x})=(\boldsymbol{\Phi}_1\circ\boldsymbol{\Phi}_0)\boldsymbol{x}.$$
For any positive integers $k$ and $G$, there exists an approximating  function $\phi(\boldsymbol{x})=(\boldsymbol{\Phi}_{1}^{G}\circ \boldsymbol{\Phi}_{0}^{G})\boldsymbol{x},$ realizable by a two-layer B-spline KAN with degree $k$ and
$G$ knots, such that the approximation error is bounded by:
\begin{equation}\label{eq:upperbound_of_BsplineKANHolder}
  \|f(\boldsymbol{x})-\phi(\boldsymbol{x}) \| \leq \lambda N(L_{G}d+1) \left( \min \left\{\frac{1}{\sqrt{2}} (k-1)^{-\frac{1}{2}},
  \sqrt{\frac{1}{12}}G^{-1} k^{\frac{1}{2}}\right\} \right)^{\alpha},
\end{equation}
where $d$ is the dimensionality of the input vector, $N$ is the count of basis functions in the second layer of the network, and $L_{G}$ is the Lipschitz constant of $\phi^{G}(\boldsymbol{x})$.
\end{theorem}
\begin{proof}
  H\"older continuous functions of order $\alpha$ with a H\"older continuous constant $\lambda$ on $[0,1]^d$ can be denoted as the class $ H_{\lambda}\left(C^{\alpha}\left([0,1]^{d}\right)\right. $, so we have $ \omega_{f}(r)\leq \lambda r^{\alpha}$.
  According to \autoref{thm:upperbound_of_Bspline}
\begin{equation*}
\begin{aligned}
\|f(\boldsymbol{x})-\phi(\boldsymbol{x}) \| &\leq N(L_{G}d+1) \cdot \omega_{f}\left(\min \left\{\frac{1}{\sqrt{2k-2}}, \sqrt{\frac{k}{12}} \cdot \frac{1}{G}\right\}\right) \\
& \leq \lambda N(L_{G}d+1) \cdot \left(\min \left\{\frac{1}{\sqrt{2k-2}}, \sqrt{\frac{k}{12}} \cdot \frac{1}{G} \right\}\right)^{\alpha}   \\
& \leq \lambda N(L_{G}d+1) \left( \min \left\{\frac{1}{\sqrt{2}} (k-1)^{-\frac{1}{2}}, \sqrt{\frac{1}{12}}G^{-1} k^{\frac{1}{2}}\right\} \right)^{\alpha},
\end{aligned}
\end{equation*}
thus, equation \eqref{eq:upperbound_of_BsplineKANHolder} holds.
\end{proof}

\begin{corollary}\label{cor:a}
 Let $f$ be a H\"older continuous functions of order $\alpha$ with a H\"older continuous constant $\lambda$ on $[-R,R]^d$, Suppose that a function $f(\boldsymbol{x})$ admits a representation
$$ f(\boldsymbol{x})=(\boldsymbol{\Phi}_1\circ\boldsymbol{\Phi}_0)\boldsymbol{x}$$
for any $k,G \in \mathbb{N}^+$, there exists a function $\phi(\boldsymbol{x})=(\boldsymbol{\Phi}_{1}^{G}\circ \boldsymbol{\Phi}_{0}^{G})\boldsymbol{x},$ implemented by a two-layer B-spline KANs with degree $k$ of the piecewise B-spline polynomial and number of knots $G$ such that
\begin{equation}\label{eq:upperbound_of_BsplineKANHolder_R}
  \|f(\boldsymbol{x})-\phi(\boldsymbol{x}) \| \leq \lambda N(L_{G}d+1)R^{\alpha} \left( \min \left\{\sqrt{2} (k-1)^{-\frac{1}{2}},
  \sqrt{\frac{1}{3}}G^{-1} k^{\frac{1}{2}}\right\} \right)^{\alpha},
\end{equation}
where $d$ denotes the dimension of the input vector, $N$ represents the number of basis functions in the second-layer network, and $L_{G}$ is the Lipschitz constant of $\phi^{G}(\boldsymbol{x})$.
\end{corollary}

\subsection{Lower bound approximation error and VC-dimension}
The approximation error and the Vapnik-Chervonenkis (VC) dimension are pivotal metrics for evaluating the capacity, or complexity, of a function class. This subsection delves into the interplay between these two fundamental concepts \cite{JordanMitchell2015}.

We commence by elucidating the definitions of VC-dimension and associated terminologies.

\begin{definition}[Shattering \cite{BartlettHarveyLiawMehrabian2017}]\label{Def:shattering}
Let $S$ represent a class of functions mapping from a domain
$\mathbb{E}$ to $\{0,1\}$. The class
$S$ is said to shatter a subset of points  $ \left\{ \boldsymbol{x_{1}} ,\boldsymbol{x_{2}}, \cdots ,\boldsymbol{x_{m}} \right\}\subseteq \mathbb{E} $, if
\begin{equation*}
\Big| \Big\{\big[s(\boldsymbol{x_{1}}),s(\boldsymbol{x_{2}}),\cdots,s(\boldsymbol{x_{m}})\big]^T\in \{0,1\}^m: s\in S\Big\}\Big|=2^m,
\end{equation*}
where $|\cdot|$ denotes the size of the set.
\end{definition}
The above definition means, given any $\xi_i\in \{0,1\}$ for $i=1,2,\cdots,m$, there exists $s\in S$ such that $s(\boldsymbol{x_{i}})=\xi_i$ for all $i$.
For a general function set  $\mathcal{A}$   with its elements mapping from $\mathbb{E}$ to $\Re$, we say  $\mathcal{A}$  shatters $ \left\{ \boldsymbol{x_{1}} ,\boldsymbol{x_{2}}, \cdots ,\boldsymbol{x_{m}} \right\}\subseteq \mathbb{E} $,  if
$\mathcal{D} \circ \mathcal{A}$ does, where
\begin{equation*}
 \mathcal{D}(t) := \left\{\begin{matrix} 1, ~ t \geq 0 , \\ 0,~ t<0 \end{matrix} \right. ~~ \text{and} ~~ \mathcal{D}\circ \mathcal{A}:=\{\mathcal{D}\circ f: f\in \mathcal{A}\} .
\end{equation*}

\begin{definition}[Growth function \cite{BartlettHarveyLiawMehrabian2017}]\label{Def:growth function}
Assume $\mathcal{A}$ is a class of functions mapping from a general domain $\mathbb{E}$ to $\{0,1\}$,
the growth function of $\mathcal{A}$ is defined as
\begin{equation*}
 \Pi_{\mathcal{A}}(m):= \max_ { \boldsymbol{x_{1}} ,\boldsymbol{x_{2}}, \cdots ,\boldsymbol{x_{m}}} \left| \left\{[h(\boldsymbol{x_{1}}),h(\boldsymbol{x_{2}}),\cdots,h(\boldsymbol{x_{m}})\big]^T \in \left\{0,1 \right\}^{m} :h\in \mathcal{A} \right\}\right| .
\end{equation*}
\end{definition}

\begin{definition}[VC-dimension \cite{BartlettHarveyLiawMehrabian2017}]\label{Def:VC-dimension}
Assume $\mathcal{A}$ is a class of functions from $\mathbb{E}$ to $\{0,1\}$.
The VC-dimension of $\mathcal{A}$, denoted by  $\operatorname{VCDim}(\mathcal{A})$, is the size of the largest shattered set, namely,
\begin{equation*}
    \operatorname{VCDim}(\mathcal{A}):= \sup \left\{m\in\mathbb{N}^+:\Pi_{\mathcal{A}}(m )=2^{m}\right\}.
\end{equation*}
If there is no largest $m$, $\operatorname{VCDim}(\mathcal{A})=\infty$.
\end{definition}
Let $\mathcal{A}$ be a class of functions from $\mathbb{E}$ to $\Re$. The VC-dimension of $\mathcal{A}$, denoted by $\operatorname{VCDim}(\mathcal{A})$, is defined by $\operatorname{VCDim}(\mathcal{A}):=\operatorname{VCDim}(\mathcal{D} \circ \mathcal{A})$
\cite{BartlettHarveyLiawMehrabian2017}.
In particular, the expression “VC-dimension of a network (architecture)” means the  VC-dimension of the function set that consists of all functions implemented by this  network architecture.

\begin{lemma}[Theorem 2.4 of \cite{ShenZhangYang2022}]\label{lem:VC&approx_holder}
	Assume $\mathcal{A}$ is a function set  with all elements defined on $[0,1]^d$. Given any $\varepsilon>0$, suppose $\operatorname{VCDim}(\mathcal{A}) \geq 1$ and
\begin{equation*}
\inf_{\phi\in \mathcal{A}}\|\phi-f\|_{L^\infty([0,1]^d)}\leq \varepsilon,\quad \text{for any} ~ f\in H_{1}(C^{\alpha}([0,1]^{d})).
\end{equation*}
Then $\operatorname{VCDim}(\mathcal{A}) \geq (9\varepsilon)^{-d/\alpha}$.
\end{lemma}
As shown in \cite{ShenYangZhang2020,LuShenYangZhang2021,ShenYangZhang2021a}, VC-dimension essentially determines the lower bound of the approximation errors of networks. Next, we introduce the relationship between two approximation errors and the VC-dimension.

\autoref{lem:VC&approx_holder} investigates the connection between VC-dimension of $\mathcal{A}$ and the  approximation errors of functions in $H_{1}(C^{\alpha}([0,1]^{d}))$ approximated by elements of $\mathcal{A}$. Denote the best approximation error of functions in $H_{1}(C^{\alpha}([0,1]^{d}))$ approximated by the elements of $\mathcal{A}$ as
\begin{equation*}
\varepsilon_{\alpha,d}(\mathcal{A}):=\sup_{f \in H_{1}(C^{\alpha}([0,1]^{d}))}\left(\inf_{\phi \in \mathcal{A}} \|\phi-f\|_{L^{\infty}([0,1]^{d})} \right).
\end{equation*}
Subsequently, by using \autoref{lem:VC&approx_holder}, we establish the inequality
$$ \operatorname{VCDim}(\mathcal{A})^{-\alpha/d}/9 \leq \varepsilon_{\alpha,d},$$
which implies that a substantial VC-dimension is essential for achieving a favorable approximation rate. This is because the optimal approximation rate is governed by a factor that is explicitly dependent on the VC-dimension. Moving forward, we shall examine the VC-dimension of the network we have constructed.

\begin{lemma}[Theorem 8 of \cite{BartlettHarveyLiawMehrabian2017}]\label{lem:VC_poly}
Consider a neural network with $M$ parameters and $U$ units with activation functions that are piecewise polynomials with at most $P$ pieces and of degree at most $d$. Let $\mathcal{A}$ be the set of (real-valued) functions computed by this network. Then $\operatorname{VCDim}((\mathcal{D} \circ \mathcal{A}))= O(MU\log((d + 1)P))$.
\end{lemma}

\begin{theorem}[Lower bound of two-layer B-spline KANs]\label{thm:upperbound_of_Bspline_KAN} Let $\boldsymbol{f}$ be a continuous functions on $[0,1]^d$, suppose that a function $\boldsymbol{f}$ admits a representation
$$ \boldsymbol{f}=(\boldsymbol{\Phi}_1\circ\boldsymbol{\Phi}_0)\boldsymbol{x}$$
for any $k,G \in \mathbb{N}^+$,there exists a function $\boldsymbol{\phi}(\boldsymbol{x})=(\boldsymbol{\Phi}_{1}^{G}\circ \boldsymbol{\Phi}_{0}^{G})\boldsymbol{x}$, implemented by a two-layer B-spline KANs with degree $k$ of the piecewise B-spline polynomial and number of knots $G$, $P= G+k-1$ is the number of B-spline basis functions used for each basis function.
The best approximation error for functions in $ H_{1}(C^{\alpha}([0,1]^{d}))$, when approximated by elements of $\mathcal{A}$, has a lower bound such that
\begin{equation}\label{eq:vc_error}
\varepsilon_{\alpha,d}(\mathcal{A})\geq C \Big
  (NP(d+1)(d+N+1)\log((d + 1)P)\Big)^
  {-\alpha/d},
\end{equation}
where $d$ denotes the dimension of the input vector, $N$ represents the number of basis functions in second-layer B-spline KANs.
\end{theorem}
\begin{proof}
According to \autoref{lem:VC_poly}
a neural network with $M$ parameters and $U$ units with activation functions that are piecewise polynomials with at most $P$ pieces and of degree at most $d$. Let $\mathcal{A}$ be the set of functions computed by this network. Then
$$\operatorname{VCDim}(\mathcal{A}) = \operatorname{VCDim}(\mathcal{D} \circ \mathcal{A})= O(MU\log\left((d + 1)P)\right)=CMU\log\left((d + 1)P)\right),$$
for our two-layer B-spline KANs, the total number of parameters is $M = (dN+N)P$, and the number of units is $ U = d+N+1$. Thus, the error in the theorem can be expressed in the following form
\begin{equation*}
\begin{aligned}
\operatorname{VCDim}(\mathcal{A}) & = \operatorname{VCDim}(\mathcal{D} \circ \mathcal{A}) = C P(dN+N)(d+N+1)\log\left((d + 1)P)\right)\\
  & \leq  C NP(d+1)(d+N+1)\log\left((d + 1)P)\right),
\end{aligned}
\end{equation*}
so we have
\begin{equation*}
\begin{aligned}
  \varepsilon_{\alpha,d}(\mathcal{A}) &
  \geq \operatorname{VCDim}(\mathcal{A})^{-\alpha/d}/9 \\
  & \geq C \Big
  (NP(d+1)(d+N+1)\log((d + 1)P)\Big)^
  {-\alpha/d},
\end{aligned}
\end{equation*}
hence, we have \eqref{eq:vc_error}.
\end{proof}

A high VC-dimension is a prerequisite for achieving a favorable approximation error; however, it is not a sufficient condition on its own. The approximation error is also contingent upon additional structural attributes of the hypothesis space $\mathcal{A}$, for instance,
$$ \operatorname{VCDim}\Big(\{\phi:\phi(x)=\cos(ax),a\in R\}\Big) = \infty. $$
However, this set fails to achieve a satisfactory approximation error when approximating H\"older continuous functions, as demonstrated in \cite{ShenYangZhang2021}. Therefore, constructing a hypothesis space with a large VC-dimension is a fundamental step towards achieving an effective approximation tool, but realizing the full potential of approximation capabilities necessitates a more nuanced design of the hypothesis space. This refined design philosophy has been central to our work in this paper, which posits that a well-crafted hypothesis space is a necessary condition for our Kolmogorov–Arnold Networks (KANs) to achieve superior approximation capabilities.

\section{Dynamical systems and discovery of dynamics}
In this section, we first present the essential notations and definitions, consistent with the conventions established in \cite{DDLMM2021}. For a more exploration of LMMs, refer to \cite{DAmbrosio2023a}. Next, we introduce a numerical scheme for an inverse problem in the discovery of dynamical systems using LMMs. Mathematically, this involves solving for the values of $\boldsymbol{f}$ given $\boldsymbol{x}$, resulting in a system of linear equations. We then analyze the numerical error associated with this system.

\subsection{LMMs: Notation and concepts}
Suppose $ d > 0$ is the dimension of the dynamics, consider the ordinary differential equation (ODE)
\begin{equation}\label{eq:IVP}
\begin{aligned}
&\frac{\mathrm{d}}{\mathrm{d}t}\boldsymbol{x}(t)=\boldsymbol{f}(\boldsymbol{x}(t)),\quad 0\leq t\leq T,\\
&\boldsymbol{x}(0)=\boldsymbol{x}_{0}.
\end{aligned}
\end{equation}
Let $ \boldsymbol{x} \in C^\infty[0,T]^d $ be an unknown vector-valued state function, $ \boldsymbol{f}: \mathbb{R}^d \rightarrow \mathbb{R}^d $ be a given vector-valued governing function, $ \boldsymbol{x}_{0} \in \mathbb{R}^d $ be a given initial vector. To seek a numerical solution, we assume a grid on the interval $[0,T]$ defined to be a set of points: $0=t_0 < t_1 < \cdots < t_{N_1} = T$ with equidistant mesh  {$t_{n+1} - t_n=h=\frac{T}{N_1}$, $n \in\{ 0, 1, \ldots, N_1\}$}. Let $[ 0,T]_h$ denote this ordered set. We denote the set of grid functions $\Gamma_h[0,T] = \left\{\boldsymbol{x} | \boldsymbol{x}\in \mathbb{R}^{(N_1+1)\times d}, \boldsymbol{x}_{n}=\boldsymbol{x} (t_n)\in \mathbb{R}^{d} ,t_n \in [0,T]_h   \right\} $, the objective for solving the initial value problem is to find an approximate value $ \boldsymbol{x}_n \approx \boldsymbol{x}(t_n) $ for each $ n $ when $ \boldsymbol{f}(\boldsymbol{x}) $ is given\cite{DAmbrosio2023a}.

An $M$-step LMM approximates the $n$-th value $\boldsymbol{x}_n = \boldsymbol{x}(t_n)$ in terms of the previous $M (M\geq 1)$ time steps $\boldsymbol{x}_{n-1}, \boldsymbol{x}_{n-2}, \ldots, \boldsymbol{x}_{n-M}$ \cite{Zarowski2004,DAmbrosio2023a}. An $M$-step linear multistep method is given by
\begin{equation}\label{eq:LMMs}
\sum_{m=0}^M \alpha_m \boldsymbol{x}_{n-m} = h \sum_{m=0}^M \beta_m f(\boldsymbol{x}_{n-m}), \quad n = M, M+1, \cdots, N_1,
\end{equation}
 where $\boldsymbol{x} \in \Gamma_h[0,T]$, the coefficients $\alpha_m, \beta_m \in \mathbb{R}$ for $m= 0, 1, \ldots, M,$ and  and $\alpha_0$ is always nonzero. By the scheme, all $\boldsymbol{x}_n$ are evaluated iteratively from $n = M$ to $n = N_1$. In each step, $\boldsymbol{x}_{n-M}, \ldots, \boldsymbol{x}_{n-1}$ are all given or computed previously such that $\boldsymbol{x}_n$ can be computed by solving algebraic equations. If $\beta_0 = 0$, the scheme is called explicit since $\boldsymbol{x}_n$ does not appear on the right-hand side and $\boldsymbol{x}_n$ can be computed directly by
\begin{equation*}
\boldsymbol{x}_n = \alpha_0^{-1} \sum_{m=1}^{M} \left(h \beta_m f(\boldsymbol{x}_{n-m}) - \alpha_m \boldsymbol{x}_{n-m}\right).
\end{equation*}
Otherwise, the scheme is called implicit and it requires solving nonlinear equations for $\boldsymbol{x}_n$. The first value $\boldsymbol{x}_0$ is simply set as $\boldsymbol{x}_0 = \boldsymbol{x}(t_0)$, while other initial values $\boldsymbol{x}_1, \ldots, \boldsymbol{x}_{M-1}$ need to be computed by other approaches before performing the LMMs if $M > 1$. Common types of LMMs include Adams-Bashforth (AB) schemes, Adams-Moulton (AM) schemes, and backwards differentiation formula (BDF) schemes.

Next, we introduce the definition of the convergence order of numerical methods, along with the corresponding definitions involved.

\begin{definition}[Residual operator \cite{DAmbrosio2023a}]\label{Def:Residual operator}
Let $\hat{\boldsymbol{x}} \in \Gamma_h[0,T]$. We define the numerical residual operator associated to (\ref{eq:LMMs}) as
\begin{equation}\label{eq:Residual operator}
(R_h \hat{\boldsymbol{x}})_n := \frac{1}{h} \sum_{m=0}^M \alpha_m \hat{\boldsymbol{x}}_{n-m} - \sum_{m=0}^M \beta_m \boldsymbol{f}(\hat{\boldsymbol{x}}_{n-m}), \quad
 n = M, M+1, \ldots, N_1.
\end{equation}
\end{definition}

\begin{definition}[Local truncation error \cite{DAmbrosio2023a}]\label{Def:LocalError}
Let $\boldsymbol{x} \in \Gamma_h[0,T]$ be the exact solution of the dynamic system (\ref{eq:IVP}) defined at the grid coordinates. The local truncation error
\begin{equation}\label{eq:LocalError}
  T(t,\boldsymbol{x};h)_n=\frac{1}{h} \sum_{m=0}^M \alpha_m \boldsymbol{x}_{n-m} - \sum_{m=0}^M \beta_m \boldsymbol{f}(\boldsymbol{x}_{n-m}), \quad
 n = M, M+1, \ldots, N_1.
\end{equation}
\end{definition}

\begin{definition}[Order of error \cite{DAmbrosio2023a}]\label{Def:ErrorOrder}
 A linear multistep method (\ref{eq:LMMs}) has error order $p$ if, for a chosen vector norm $\left\|\cdot\right\|$, there exists a real constant $C >0$ such that
 $$\| \boldsymbol{T}(t,\boldsymbol{x};h)\|\leq Ch^p,$$
where $\boldsymbol{T}(t,\boldsymbol{x};h)=
\Big(T(t,\boldsymbol{x};h)_M, T(t,\boldsymbol{x};h)_{M+1}, \ldots, T(t,\boldsymbol{x};h)_{N_1}\Big) \in \mathbb{R}^{(N_1-M+1)\times d}$ and $C$ is independent on $t, \boldsymbol{x}, h$.
\end{definition}

\subsection{Discovery of dynamics}
The discovery of dynamics can be viewed as an inverse process of solving a dynamical system defined by (\ref{eq:IVP}). Given information on the state $\boldsymbol{x}$ at equidistant time steps $ \{t_n\}_{n=0}^{N_1} $, our goal is to recover $ \boldsymbol{f} $, the governing function of the state.

Let $\boldsymbol{x}(t) \in C^{\infty}([0, T])^d $ and $ \boldsymbol{f}(\cdot): \mathbb{R}^d \rightarrow \mathbb{R}^d $ be two vector-valued functions that are both unknown. Given the observations $ \boldsymbol{x}_n = \boldsymbol{x}(t_n) $ for $ n = 0, \ldots, N $, we aim to determine $ \boldsymbol{f}(\cdot) $. An effective method is to establish a discrete relationship between $ \boldsymbol{x}_n $ and $ \boldsymbol{f}_n \approx \boldsymbol{f}(\boldsymbol{x}_n) $ using LMMs \cite{DDLMM2021}, expressed as:
\begin{equation}\label{eq:LMMs_Discovery}
h \sum_{m=0}^{M} \beta_m \boldsymbol{f}_{n-m} = \sum_{m=0}^{M} \alpha_m \boldsymbol{x}_{n-m}, \quad n = M, M+1, \ldots, N_1.
\end{equation}
Here, $ \boldsymbol{f}_n \in \mathbb{R}^d $ serves as an approximation of $ \boldsymbol{f}(\boldsymbol{x}_n) $. This discovery process is essentially the inverse of solving the dynamical system. To simplify, we focus on a scalar system, resulting in the following equation:
\begin{equation}\label{eq:LMMs_Discovery1}
h \sum_{m=0}^{M} \beta_m f_{n-m} = \sum_{m=0}^{M} \alpha_m x_{n-m}, \quad n = M, M+1, \ldots, N_1.
\end{equation}
It is important to note that $ f_n $ may not be present in (\ref{eq:LMMs_Discovery1}) for some indices $ n $ between $0$ and $N_1$. For instance, in AB schemes, $ f_{N_1} $ is not included in (\ref{eq:LMMs_Discovery1}) because $ \beta_0 = 0 $. Generally, for a given LMM, we define $r$ and $q$ as the first and last indices where $ f_r $ and $ f_{q} $ appear in (\ref{eq:LMMs_Discovery1}) with non-zero coefficients (i.e., $ \beta_{M-r}$ and $\beta_{N_1-q}$ are both non-zero). We denote $\tau :=q-r+1$ as the total number of $f_n$ involved in (\ref{eq:LMMs_Discovery1}).
For each linear $M$-step method, it is supposed to compute all unknowns $\{f_n\}_{n=r}^{q}$ by the linear relation (\ref{eq:LMMs_Discovery1}). We write
{\setlength{\arraycolsep}{2pt}
\begin{equation*}
\overrightarrow{\boldsymbol{f}_h} := \begin{bmatrix} f_r & f_{r+1} & \cdots & f_{q} \end{bmatrix}^{\mathrm{T}} \in \mathbb{R}^{\tau},
\end{equation*}
\begin{equation*}
\overrightarrow{\boldsymbol{b}_h} := \frac{1}{h} \left[ \sum_{m=0}^{M} \alpha_m x_{M-m} \quad \sum_{m=0}^{M} \alpha_m x_{M+1-m} \quad \cdots \quad \sum_{m=0}^{M} \alpha_m x_{N_1-m} \right]^{\mathrm{T}} \in \mathbb{R}^{N_1-M+1},
\end{equation*}
\begin{equation*}
\boldsymbol{B}_h := \begin{bmatrix}
\beta_{M-r} & \beta_{M-r-1} & \cdots & \beta_{N_1-q} & & \\
& \beta_{M-r} & \beta_{M-r-1} & \cdots & \beta_{N_1-q} & & \\
&& \ddots & \ddots & \ddots & \ddots & \\
&&& \beta_{M-r} & \beta_{M-r-1} & \cdots & \beta_{N_1-q}
\end{bmatrix} \in \mathbb{R}^{(N_1-M+1) \times \tau}.
\end{equation*}}

Then (\ref{eq:LMMs_Discovery}) leads to the following linear system
\begin{equation}\label{eq:linear_system}
\boldsymbol{B}_h \overrightarrow{\boldsymbol{f}_h} = \overrightarrow{\boldsymbol{b}_h}.
\end{equation}

In (\ref{eq:LMMs_Discovery}), the number of equations may not match the number of unknowns. For AB and AM schemes, it is insufficient to determine $ \{f_n\}_{n=r}^{q} $ because there are fewer equations than unknowns. This results in the linear system (\ref{eq:linear_system}) being underdetermined.

To address this issue, we can introduce $ \Omega_a:=q-(N_1-M+1) $ auxiliary linear conditions to ensure $ \{f_n\}_{n=r}^{q} $ is uniquely determined. For instance, we can compute $ \Omega_a $ unknown $ f_n $ values directly using a first-order (derivative) finite difference method (FDM) based on related data. To maintain consistency, the chosen FDM should have the same error order as the LMM. Assuming the LMM has order $ p $, a straightforward approach is to calculate the initial $ \Omega_a $ unknowns using a FDM of order $p$,
\begin{equation}\label{eq:LMMs_DiscoveryMatix}
f_n = \frac{1}{h} \sum_{m=0}^{p} \mu_m x_{n+m}, \quad n = r, r+1, \ldots, r + \Omega_a - 1,
\end{equation}
where $\mu_m$ are the corresponding finite difference coefficients. Note that (\ref{eq:LMMs_DiscoveryMatix}) has the error estimate
\begin{equation*}
\max_{r \leq n \leq r + \Omega_a - 1} \left| f_n - f(\boldsymbol{x}(t_n)) \right| = O(h^p), \quad \text{as } h \rightarrow 0.
\end{equation*}
Then combining (\ref{eq:LMMs_Discovery}) and (\ref{eq:LMMs_DiscoveryMatix}) leads to the augmented linear system
\begin{equation}\label{eq:LMMs_DiscoveryMatixA}
\boldsymbol{A}_h \overrightarrow{\boldsymbol{f}_h} = \left[ \begin{array}{c} \boldsymbol{c}_h \\ \overrightarrow{\boldsymbol{b}_h} \end{array} \right],
\end{equation}
where
\begin{equation*}
\boldsymbol{c}_h := \frac{1}{h} \left[ \sum_{m=0}^{p} \mu_m x_{s+m} \quad \sum_{m=0}^{p} \mu_m x_{s+1+m} \quad \cdots \quad \sum_{m=0}^{p} \mu_m x_{s+\Omega_a-1+m} \right]^T \in \mathbb{R}^{\Omega_a},
\end{equation*}
\begin{equation*}
\boldsymbol{A}_h := \left[ \begin{array}{c} \boldsymbol{C} \\ \boldsymbol{B}_h \end{array} \right] \quad \text{and} \quad \boldsymbol{C} := \left[ \begin{array}{ll} \boldsymbol{I}_{\Omega_a} & \boldsymbol{O} \end{array} \right]
\end{equation*}
with $\boldsymbol{I}_{\Omega_a}$ being the $\Omega_a \times \Omega_a$ identity matrix and $\boldsymbol{O}$ being the zero matrix of size $\Omega_a \times (\tau - \Omega_a)$. Clearly, (\ref{eq:LMMs_DiscoveryMatixA}) has a unique solution since the coefficient matrix is lower triangular with nonzero diagonals. Moreover, if $M \ll N_1$, the linear system (\ref{eq:LMMs_DiscoveryMatixA}) is sparse.

In general, as noted in \cite{DDLMM2021}, we can define auxiliary conditions in various ways beyond what has been discussed. Different types of auxiliary conditions, such as initial and terminal conditions\cite{GulgecShiDeshmukhPakzadEtAl2019}, can affect the stability and convergence of the method, as further discussed in \cite{DDLMM2021}. A key question is whether the regularization effect from neural network approximations could help alleviate these issues.

\section{Convergence Analysis}
In this section, we first present the approach of using neural networks to approximate the linear system derived from the discretized LMMs scheme (\ref{eq:linear_system}), as well as the linear system obtained after incorporating auxiliary conditions. We then analyze the relationship between the error of the learned governing function $\boldsymbol{f}_{\mathcal{NN}}$ and the original governing function $\boldsymbol{f}$, which consists of both the discretization error of the numerical scheme and the network approximation error. Finally, we trace the error in the governing function back to the error in the original state function $\boldsymbol{x}$.

\subsection{Network-based discovery method}
Let us revisit the dynamics discovery on a single trajectory, as introduced in Section 3. Conventional LMMs are straightforward to implement, yielding solutions by solving a linear system. However, this approach only computes the governing function \(\boldsymbol{f}\) at fixed, equidistant time steps, leaving the relationship between \(\boldsymbol{f}\) and the state \(\boldsymbol{x}\) unknown. One way to address this is to approximate each component of \(\boldsymbol{f}\) with structured functions, such as neural networks, polynomials, or splines allowing closed-form expressions for \(\boldsymbol{f}\) through optimization. Once \(\boldsymbol{f}\) is explicitly recovered, future states \(\boldsymbol{x}\) on the same or nearby trajectories can be predicted by solving (\ref{eq:IVP}) with initial or perturbed conditions.

Neural networks are especially effective for this purpose, particularly when \(d\) is moderately large, as they handle high-dimensional inputs better than other structures. Thus, we focus on network-based methods here, though the approach generalizes to other types of approximations.

We consider neural network approximations based on the LMMs scheme (\ref{eq:LMMs_Discovery1}). Generally, We use $\mathcal{NN}(k, G, N)$ to denote the set of all two-layer B-spline KANs with $k$-degree, $G$ nodes, and $N$ represents the number of basis
functions in second layer. Now we introduce a network $f_{\mathcal{NN}}(\boldsymbol{x})\in \mathcal{NN}(k, G, N)$ to approximate $f(\boldsymbol{x})$ an arbitrary component of  $\boldsymbol{f}(\cdot)$. The neural network method can be developed by
\begin{equation}\label{eq:NN_discovery}
h\sum_{m=0}^M\beta_m f_{\mathcal{NN}}(\boldsymbol{x}_{n-m})=\sum_{m=0}^M\alpha_m \boldsymbol{x}_{n-m},\quad n=M,M+1,\ldots,N_1,
\end{equation}
where $\boldsymbol{x}_{n}$ for $n=0, \ldots, N_1$ are given sample locations.

Unfortunately, if the VC-dimension of $\mathcal{NN}(k, G, N)$ is very small, then there does not exist an $f_{\mathcal{NN}} \in \mathcal{NN}(k, G, N)$ that can exactly satisfy (\ref{eq:NN_discovery}). Therefore, we typically seek $f_{\mathcal{NN}} \in \mathcal{NN}(k, G, N)$ by minimizing the residual of (\ref{eq:NN_discovery}) within a machine learning framework, such that
\begin{equation}\label{eq:min_NN}
 J_h(f_{\mathcal{NN}})=\min_{u \in \mathcal{NN}(k, G, N)}J_h(u),
\end{equation}
where
\begin{equation}\label{eq:Jh}
J_h(u):=\frac1{N_1-M+1}\sum_{n=M}^{N_1}\left|\sum_{m=0}^M\beta_mu(\boldsymbol{x}_{n-m})-\sum_{m=0}^Mh^{-1}\alpha_mx_{n-m}\right|^2.
\end{equation}
However, similar to the underdetermined linear system (\ref{eq:linear_system}) which has an infinite number of solutions, ensuring a unique minimizer at the grid points necessitates additional conditions. The conditions are introduced and an augmented loss function is constructed based on equation (\ref{eq:Jh}) to fulfill the objective, as detailed in the work \cite{DuGuYangZhou}.

The auxiliary minimization problem is formulated to find the optimal neural network function $f_{\mathcal{NN}}$
within the space of neural networks $\mathcal{NN}(k, G, N)$, as depicted in Equation \eqref{eq:auxiliarymin_NN}:
\begin{equation}\label{eq:auxiliarymin_NN}
 J_{a,h}(f_{\mathcal{NN}})=\min_{u \in \mathcal{NN}(k, G, N)}J_{a,h}(u).
\end{equation}
The objective function $J_{a,h}(u)$
is defined as the sum of squared errors over a training dataset, which includes both the approximation error and the collocation error, as shown in Equation \eqref{eq:Jah}:

\begin{equation} \label{eq:Jah}
J_{a,h}(u):=\frac{1}{\tau} \Bigg(\sum_{n=s}^{s+N_a-1}
\left|u(\boldsymbol{x}_n)-\frac{1}{h}\sum_{m=0}^p\mu_mx_{n+m}\right|^2
 +\sum_{n=M}^{N_1}\left|\sum_{m=0}^M\beta_mu(\boldsymbol{x}_{n-m})-\sum_{m=0}^Mh^{-1}
\alpha_mx_{n-m}\right|^2\Bigg).
\end{equation}
This enhanced optimization strategy ensures the uniqueness of the minimizer at the grid points within the function space. For further details on the network implementation process, refer to Figure \ref{fig:KAN_LMM}.
\begin{figure}[htp!]
  \centering
  \includegraphics[width=1\linewidth]{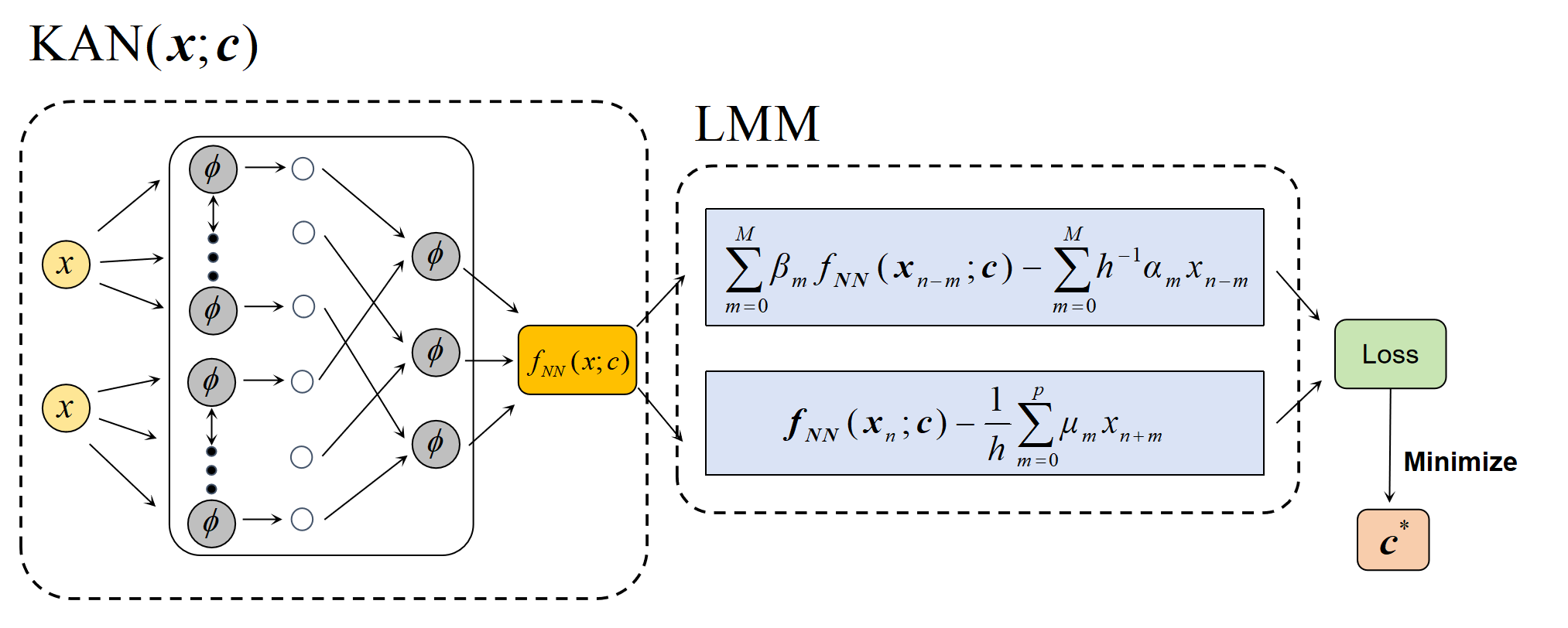}
  \caption{Schematic representation of the network architecture employing LMMs in conjunction with KANs for dynamical system discovery.}
  \label{fig:KAN_LMM}
\end{figure}
\subsection{Error estimates}
We consider the error estimation of the discovery on the trajectory $\Delta:=\{\boldsymbol{x}(t):0\leq t\leq T\}.$

\begin{definition}[$l^{2}$ seminorm \cite{DuGuYangZhou}]\label{Def:L2seminorm}
Let $\Delta:=\{\boldsymbol{x}(t):0\leq t\leq T\}$, for any $f \in C(\Delta)$ with a given $h >0$, the $l^{2}$ seminorm as
\begin{equation}\label{eq:L2seminorm}
  |f|_{2,h}:=\left(\frac{1}{N_1+1}\sum_{n=0}^{N_1}|f(\boldsymbol{x}_n)|^2\right)^{1/2}.
\end{equation}
\end{definition}

As discussed above, for a specific LMM, some states in $\{\boldsymbol{x}_n\}_{n=0}^{N_1}$ may not be involved in the  scheme. For fairness, we study the convergence at all involved states $\{\boldsymbol{x}_n\}_{n=r}^{q}$. Therefore, we rewrite $|f|_{2,h}$ as the LMM-related seminorm$$|f|_{2,h}:=\left(\frac{1}{ \tau}\sum_{n=r}^{q}|f(\boldsymbol{x}_n)|^2\right)
^{1/2} \quad \text{for all}~f \in C(\Delta).$$
Next, we represent \(\{f(\boldsymbol{x}_n)\}_{n=r}^{q}\) as the vector \(\vec{\boldsymbol{f}} := \left[f(\boldsymbol{x}_r) ~ f(\boldsymbol{x}_{r+1})~ \ldots ~ f(\boldsymbol{x}_{q})\right]^T\) and aim to estimate the distance between \(f_{\mathcal{NN}}\) and \(f\). Here, \(f_{\mathcal{NN}}\) denotes the minimizer of (\ref{eq:auxiliarymin_NN}), the result of our network optimization. To evaluate these errors, we account for both the numerical discretization error and the network approximation error, and we present their relationship below.

\begin{lemma}[Theorem 5.1 of \cite{DuGuYangZhou}]\label{lem:error_f_fnn}
	In the dynamical system (\ref{eq:IVP}), suppose $\boldsymbol{x}(t) \in C^{\infty}([0, T])^d $ and $\boldsymbol{f}$ is defined in $\Delta '$, a small neighborhood of $\Delta$. Let $f$ be an arbitrary component of $\boldsymbol{f}$ and $h := T /N_1$ with an integer $N_1 > 0$, then
\begin{equation}\label{eq:error_f_fnnA}
\left|f_{\mathcal{A}}-f\right|_{2,h}<C\kappa_2
(\boldsymbol{A}_h)\left(h^p+e_\mathcal{A}\right),
\end{equation}
where $f_{\mathcal{A}} \in \mathcal{A}$ is a global minimizer of $J_{a,h}$ defined by (\ref{eq:Jah}) corresponding to an LMM with order $p$ and $\kappa_2
(\boldsymbol{A})=
\|\boldsymbol{A}\|_2\|\boldsymbol{A}^{-1}\|_2$.
\end{lemma}
\begin{lemma}[Theorem 5.4 of \cite{DuGuYangZhou}]\label{lem:k2_bound}
Let $\boldsymbol{A}_h$ be the matrix defined by \eqref{eq:LMMs_DiscoveryMatixA}, and $p_{h}(z)$ be the following polynomial
\begin{equation*}
  p_{h}(z)= \sum_{i=N-q}^{M-r}\beta_{i}z^{M-r-i}.
\end{equation*}
If all roots of $p_{h}(z)$ have modulus smaller than 1, then $\kappa_2
(\boldsymbol{A}_h)$ is uniformly bounded with respect to $N_1$.
\end{lemma}

\begin{theorem}\label{thm:error_f_fnn}
In the dynamical system (\ref{eq:IVP}), suppose $\boldsymbol{x}(t) \in C^{\infty}([0, T])^d $ and $\boldsymbol{f}$ is defined in $\Delta '$, a small neighborhood of $\Delta$. Let $f$ be an arbitrary component of $\boldsymbol{f}$. Also, let $N_1 > 0$ be an integer and $h := T /N_1$ , then we have
\begin{equation}\label{eq:error_f_fnn}
\left|f_{\mathcal{NN}}-f\right|_{2,h}<C\kappa_2
(\boldsymbol{A}_h)\left(h^p+e_\mathcal{NN}(k, G, N) \right),
\end{equation}
with $e_\mathcal{NN}(k, G, N) = 2R_{\Delta'}N(L_{G}d+1) \cdot \omega_{f}\left(\min \left\{\frac{1}{\sqrt{2k-2}}, \sqrt{\frac{k}{12}} \cdot \frac{1}{G} \right\}\right)$, where $f_{\mathcal{NN}} \in \mathcal{NN}(k, G, N)$ is a global minimizer of $J_{a,h}$ defined by (\ref{eq:Jah}) corresponding to an LMM with order $p$; $C$ is a constant independent of $h, k, G, N$. In particular, if $\kappa_2
(\boldsymbol{A}_h)$ is uniformly bounded for all $h > 0$, then
\begin{equation*}
  \lim_{k,G,N \rightarrow \infty ,h\rightarrow 0}| f_{\mathcal{NN}}-f|_{2,h}=0.
\end{equation*}
\end{theorem}

\begin{proof}
According to \autoref{lem:error_f_fnn}, if we replace the function approximation set with $\mathcal{NN}(k, G, N)$, we obtain the following expression
\begin{equation*}
\left|f_{\mathcal{NN}}-f\right|_{2,h}<C\kappa_2
(\boldsymbol{A}_h)\left(h^p+e_\mathcal{NN}(k, G, N) \right).
\end{equation*}

By applying \autoref{cor:upperbound_of_Bspline_R}, we substitute $e_\mathcal{A}$ with the error $e_\mathcal{NN}(k, G, N) = 2R_{\Delta'}N(L_{G}d+1) \cdot \omega_{f}\left(\min \left\{\frac{1}{\sqrt{2k-2}}, \sqrt{\frac{k}{12}} \cdot \frac{1}{G} \right\}\right)$ of our two-layer B-spline KANs, where $R_{\Delta'}$ denotes the radius of the domain of the function we aim to approximate. Thus, the total error is given by

 \begin{equation*}
 \left|f_{\mathcal{NN}}-f\right|_{2,h}  < C\kappa_2
 (\boldsymbol{A}_h)\left(h^p+2R_{\Delta'}N(L_{G}d+1) \cdot  \omega_{f}\left(\min \left\{\frac{1}{\sqrt{2k-2}}, \sqrt{\frac{k}{12}} \cdot \frac{1}{G}\right\}\right) \right).
 \end{equation*}
 
where uniform boundedness of $\kappa_2
(\boldsymbol{A}_h)$ here is guaranteed by \autoref{lem:k2_bound}.
\end{proof}

Now that we have derived the approximation error of $ \boldsymbol{f}_{\mathcal{NN}} $ for approximating $\boldsymbol{f}$ , we proceed to analyze the error relationship between $ \boldsymbol{x}_{\mathcal{NN}}$, obtained by solving with the approximated $ \boldsymbol{f}_{\mathcal{NN}} $, and the exact solution $\boldsymbol{x}$ derived from the original function.

\begin{theorem}\label{thm:error_x_xnn}
Suppose $\boldsymbol{x}(t), \boldsymbol{x}_{\mathcal{NN}}(t) \in C^{\infty}([0, T])^d $ and $\boldsymbol{f}, \boldsymbol{f}_{\mathcal{NN}}$ is defined in $\Delta'$, a small neighborhood of $\Delta$, in the dynamical system (\ref{eq:IVP}), If the error between $\boldsymbol{f}$ and $\boldsymbol{f}_{\mathcal{NN}}$ is $\boldsymbol{\varepsilon}$, then the corresponding error between the solutions $\boldsymbol{x}(t)$ and $\boldsymbol{x}_{\mathcal{NN}}(t)$ is given by
\begin{equation}\label{eq:error_x_xnn}
\| \boldsymbol{x}(t)- \boldsymbol{x}_{\mathcal{NN}}(t) \| \leq \Big(\| \boldsymbol{x}(0)- \boldsymbol{x}_{\mathcal{NN}}(0) \|+ \boldsymbol{\varepsilon}t \Big)e^{L_{f}t}.
\end{equation}
\end{theorem}
\begin{proof}
Let $\boldsymbol{z}(t)=\boldsymbol{x}(t)- \boldsymbol{x}_{\mathcal{NN}}(t)$,
\begin{equation*}
\begin{aligned}
\frac{d \boldsymbol{z}(t)}{dt} & = \frac{d\boldsymbol{x}(t)}{dt} - \frac{d\boldsymbol{x}_{\mathcal{NN}}(t)}{dt} \\
& = \boldsymbol{f}(t, \boldsymbol{x}(t)) - \boldsymbol{f}_{\mathcal{NN}}(t, \boldsymbol{x}_{\mathcal{NN}}(t))) \\
& = \boldsymbol{f}(t, \boldsymbol{x}(t)) - (\boldsymbol{f}(t, \boldsymbol{x}_{\mathcal{NN}}(t)) + \boldsymbol{\varepsilon}) .
\end{aligned}
\end{equation*}
Integrating both sides of the above equation yields
\begin{equation*}
\begin{aligned}
\boldsymbol{z}(t) & = \boldsymbol{z}(0) + \int_0^t \Big(\boldsymbol{f}(s, \boldsymbol{x}(s)) -  \boldsymbol{f}(s, \boldsymbol{x}_{\mathcal{NN}}(s))\Big) ds + \int_0^t \boldsymbol{\varepsilon}ds   \\
 & = (\boldsymbol{z}(0) + \boldsymbol{\varepsilon}t) +\int_0^t \Big(\boldsymbol{f}(s, \boldsymbol{x}(s)) -  \boldsymbol{f}(s, \boldsymbol{x}_{\mathcal{NN}}(s))\Big) ds.
\end{aligned}
\end{equation*}
Taking norms on both sides of the above equation, one has
\begin{equation*}
\begin{aligned}
\| \boldsymbol{z}(t)\|  & \leq \| \boldsymbol{z}(0) + \boldsymbol{\varepsilon}t \| + \left\| \int_0^t \Big(\boldsymbol{f}(s, \boldsymbol{x}(s)) -  \boldsymbol{f}(s, \boldsymbol{x}_{\mathcal{NN}}(s))\Big) ds \right\|  \\
& \leq \| \boldsymbol{z}(0) + \boldsymbol{\varepsilon}t \| + L_{f} \int_0^t \left\| \boldsymbol{z}(s) \right\| ds.
\end{aligned}
\end{equation*}
It follows from Gronwall's inequality that
\begin{equation*}
\| \boldsymbol{z}(t)\| \leq \| \boldsymbol{z}(0) + \boldsymbol{\varepsilon}t \|e^{L_{f}t},
\end{equation*}
which implies \eqref{eq:error_x_xnn}. This completes the proof.
\end{proof}

\section{Numerical experiments}
This section presents a series of numerical experiments to substantiate our theoretical findings on uncovering hidden dynamics utilizing Linear Multistep Methods (LMMs) in conjunction with Kolmogorov-Arnold Networks (KANs). Drawing on precedents set by \cite{DuGuYangZhou,DDLMM2021,RaissiPerdikarisKarniadakis2018}, we examine three distinct examples: a simple linear ODE system, a glycolytic oscillator, and a model of opinion dynamics.

We assess the error orders of Adams-Bashforth (AB), Backward Differentiation Formula (BDF), and Adams-Moulton (AM) schemes within this context, with detailed coefficients referenced in \cite{DAmbrosio2023a}.

The experimental framework is outlined as follows:
\begin{itemize}
\item \textbf{Optimizer and Hyperparameters:} The optimization of network-based models is efficiently conducted using the Adam subroutine, accessible within the PyTorch library, which implements the algorithm detailed in \cite{KingmaBa2014}. Batch gradient descent is applied with a learning rate of 0.01.

\item \textbf{Network Configuration:} We engage two-layer KANs with B-spline activation functions for approximation purposes, initializing weights according to the Xavier uniform distribution as described in \cite{GlorotBengio2010}. Guided by the network error component within our theoretical results, we experiment with parameters  $k$ and $G$ to identify their  effective values.

\item \textbf{Data Generation:}  In subsequent examples, data generation is performed by solving dynamical systems numerically with \texttt{scipy.integrate.odeint}, which automatically switches between Adams methods for non-stiff problems and BDF for stiff problems. The solver uses \(\texttt{rtol}=10^{-13}\) and \(\texttt{atol}=10^{-13}\) to ensure high precision. By basing our dynamical discoveries on such precisely generated data, we can more accurately assess the validity and reliability of the vector field functions obtained through our fitting process.
\end{itemize}

\subsection{Linear ode system}
Let us consider the following model problem
\begin{equation}\label{1}
\left\{
\begin{aligned}
&\frac{dx}{dt}=2x+3y,\\
&\frac{dy}{dt}=-4y,
\end{aligned}\right.
\end{equation}
with the initial state \([x_0, y_0]^{\mathrm{T}} = [0, 1]^{\mathrm{T}}\). The state of the system is explicitly given by \(x(t) = \frac{1}{2}e^{2t} - \frac{1}{2}e^{-4t}\) and \(y(t) = e^{-4t}\). 

\begin{figure}[htp!]
\centering
\subfloat{\includegraphics[width=0.46\linewidth]{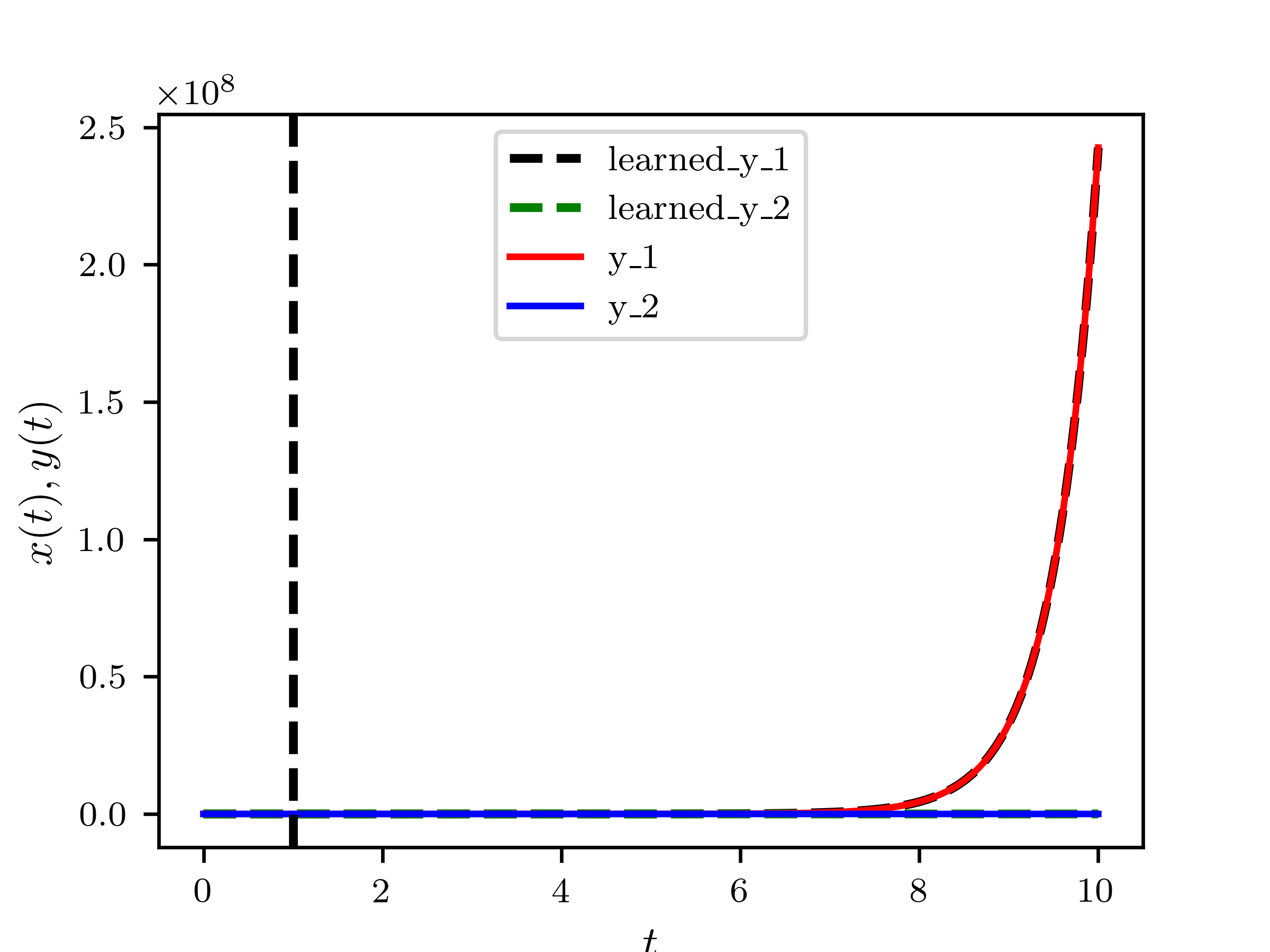}}
\quad
\subfloat{\includegraphics[width=0.46\linewidth]{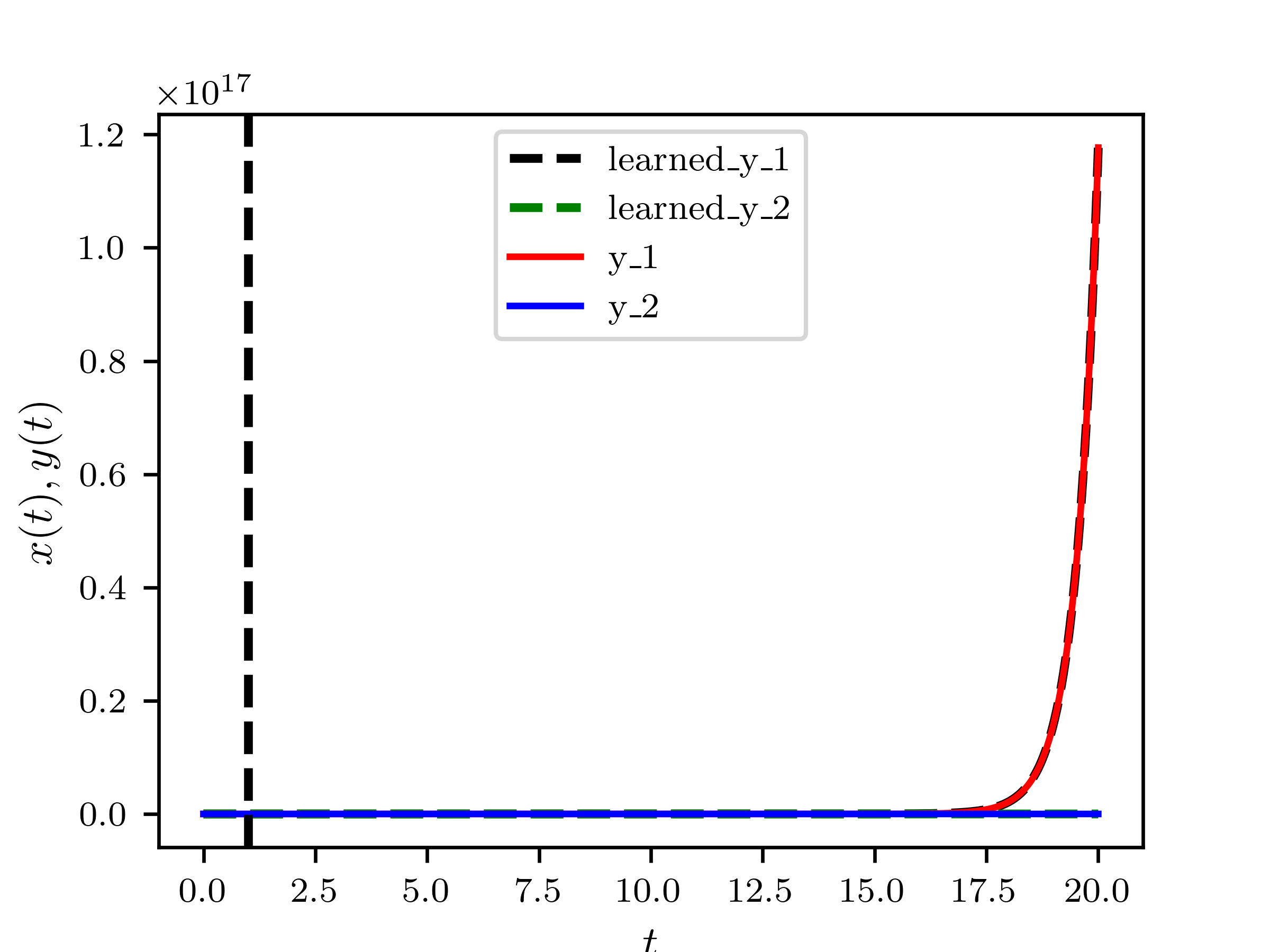}}
\caption{ Left panel: The black dashed vertical line at $t=1$ demarcates the boundary between the training interval $[0,1]$ and the prediction interval $[1,10]$. Right panel: The black dashed vertical line at $t=1$ demarcates the boundary between the training interval $[0,1]$ and the extended prediction interval
$[1,20]$.}
\label{fig:x_xnn1}
\end{figure}
\begin{figure}[htp!]
  \centering
  \includegraphics[width=0.8\linewidth]{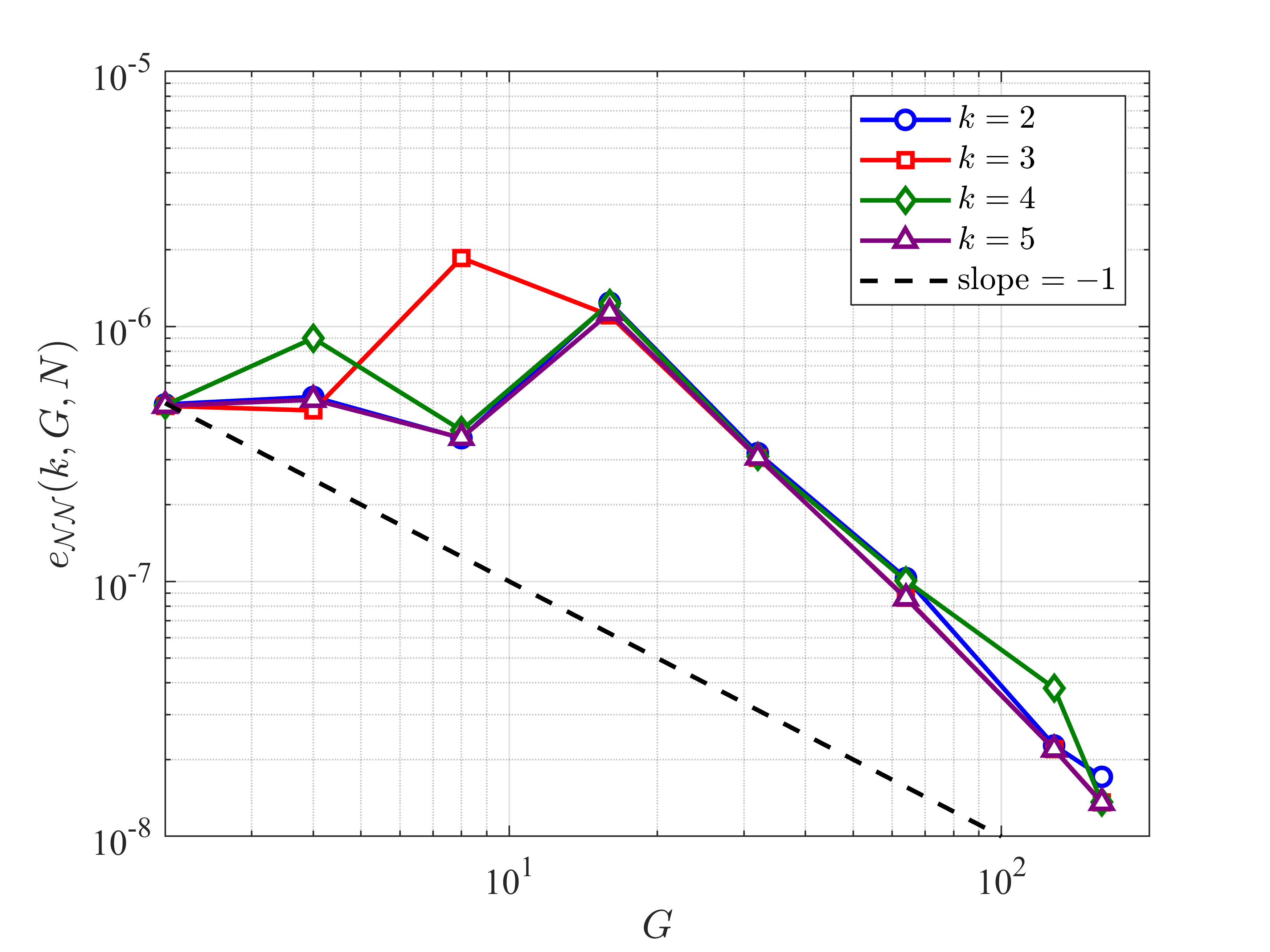}
  \caption{Influence of B-Spline KANs parameters: $e_\mathcal{NN}(k, G, N)$ as a function of degree $k$ and node number $G$.}
  \label{fig:ode_Error_kG}
\end{figure}
Firstly, we design an experiment to investigate the approximation capabilities of various multistep methods under identical network conditions. We systematically compare their performance, with numerical results detailed in Table 1. Notably, the second-order AM 1-step method outperforms others, achieving the lowest trajectory error and demonstrating superior stability, corroborating findings from \cite{DuGuYangZhou,RaissiPerdikarisKarniadakis2018}.
Subsequently, we focus on the impact of parameters  $k$ and $G$  of the B-spline basis functions on the KAN network error $e_\mathcal{NN}(k, G, N)$ under the AM 1-step scheme.
Experiments to generate \autoref{fig:ode_Error_kG} are conducted with B-spline degrees
$k=2,3,4,5$ and number of nodes
$G=4,8,16,32,64,128,160$ over the interval $[0,10]$ with a step size
$h=0.001$ for the AM scheme. The results demonstrate that when $k$ is relatively small, the parameter $G$
predominantly influences the error $e_\mathcal{NN}(k, G, N)$. \autoref{fig:ode_Error_kG} demonstrates that with the increase of $G$ towards infinity, the convergence rate of the network
 $\boldsymbol{f}_{\mathcal{NN}}$ approximating the vector field
$\boldsymbol{f}$ with respect to
$1/G$ exceeds 1. 
\begin{table}[H]
\centering
\caption{Error between $\boldsymbol{f}$ and $\boldsymbol{f}_{\mathcal{NN}}$
  for multistep methods with varying step numbers $M$ using a two-layer B-Spline KANs architecture.}
\label{tab:f_fnn_M}
\begin{tabular}{ | c | c | c | c  | c | c | c | }\hline
M-step   & 1-step & 2-step & 3-step & 4-step & 5-step & 6-step \\ \hline 
AB & 8.71e-08 & 3.46e-07 & 7.80e-07 & 1.39e-06 & 2.16e-06 &
3.12e-06\\
AM & 8.60e-08 & 3.44e-07 & 7.75e-07 & 2.15e-06 & 2.16e-06 &  3.11e-06 \\
BDF & 8.61e-08  & 1.54e-07 & 2.32e-07 & 3.20e-07 & 4.16e-07 & 5.21e-07 \\ \hline
\end{tabular}
\end{table}
\begin{figure}[htp!]
  \centering
  \includegraphics[width=0.8\linewidth]{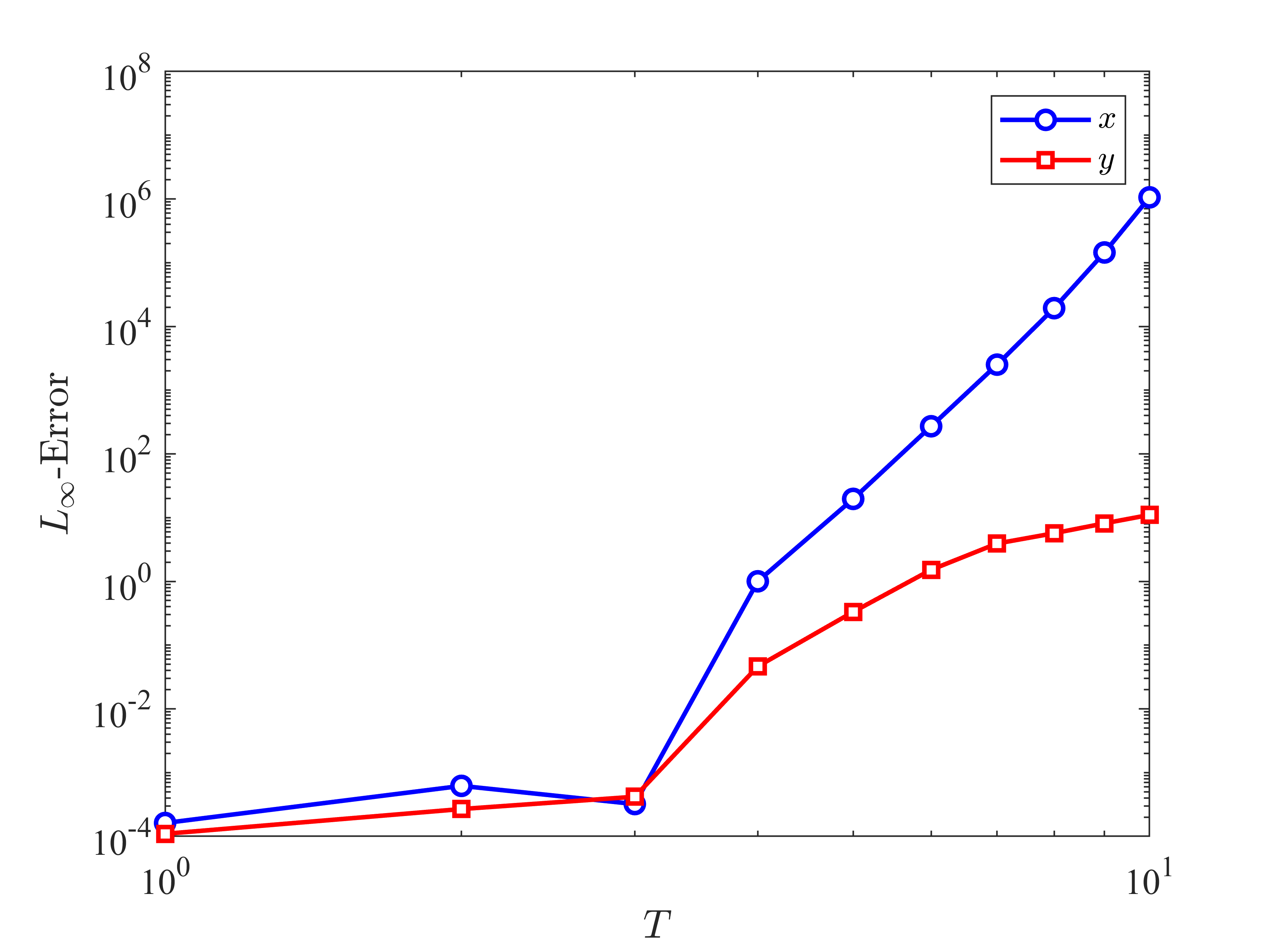}
  \caption{The $L^\infty$-norm error between $\boldsymbol{x}(t)$ and  $\boldsymbol{x}_{\mathcal{NN}}(t)$ over the interval $[0,T].$}
  \label{fig:xnn_t}
\end{figure}
Next, we examine the \autoref{thm:error_x_xnn}. In our numerical experiment, we utilize the previously determined network configuration, specified by the parameters $k=3,$ $G=64$, $h=10^{-3}$, and 2200 iterations with the AM 1-step method. We calculate the errors between the neural network approximation $\boldsymbol{x}_{\mathcal{NN}}$
  and the original function $\boldsymbol{x}$ for $T=1,2, \ldots, 10$. These results are depicted in \autoref{fig:xnn_t} using a log-log plot, showing that the errors for both solutions
$x$ and $y$ increase exponentially with $T$, consistent with our theoretical predictions.


\subsection{Glycolytic oscillator}
This subsection presents a simulation of the glycolytic oscillator model, exemplifying the complex nonlinear dynamics inherent in biological systems. The model is formulated by a system of ordinary differential equations that describe the concentrations of various species:
\begin{equation}\label{4}
\left\{
\begin{aligned}
&\frac{dS_1}{dt} = J_0 - \frac{k_1 S_1 S_6}{1 + \left(S_6 / K_1\right)^q},\\
&\frac{dS_2}{dt} = 2 \frac{k_1 S_1 S_6}{1 + \left(S_6 / K_1\right)^q} - k_2 S_2 (N - S_5) - k_6 S_2 S_5, \\
&\frac{dS_3}{dt} = k_2 S_2 (N - S_5) - k_3 S_3 (N - S_6),\\
&\frac{dS_4}{dt} = k_3 S_3 (A - S_6) - k_4 S_4 S_5 - \kappa (S_4 - S_7),\\
&\frac{dS_5}{dt} = k_2 S_2 (N - S_5) - k_4 S_4 S_5 - k_6 S_2 S_5,\\
&\frac{dS_6}{dt} = -2 \frac{k_1 S_1 S_6}{1 + \left(S_6 / K_1\right)^q} + 2 k_3 S_3 (A - S_6) - k_5 S_6,\\
&\frac{dS_7}{dt} = \psi \kappa (S_4 - S_7) - k_7 S_7
\end{aligned}\right.
\end{equation}
with initial condition $\mathcal{S}_0 =[1.125,0.95,0.075,0.16,0.265,0.7,0.092]$. Model parameters are detailed in \autoref{tab: parameters}.

\begin{table}[H]
\centering
\setlength{\tabcolsep}{10pt}
\renewcommand{\arraystretch}{1.2}
\caption{Parameters of the glycolytic oscillator}\label{tab: parameters}
\begin{tabular}{|c|c|c|c|c|c|c|}
\hline
$k_1$ & $k_2$ & $k_3$ & $k_4$ & $k_5$ & $k_6$ & $k_7$ \\
\hline
100.0 & 6.0 & 16.0 & 100.0 & 1.28 & 12.0 & 1.8  \\
\hline
$J_0$ & $\kappa$ & $q$ & $K_1$ & $\psi$ & $N$ & $A$ \\
\hline
2.5 & 13.0 & 4 & 0.52 & 0.1 & 1.0 & 4.0 \\
\hline
\end{tabular}
\end{table}

Employing the same network configuration as in the previous example ($k = 3$ , $G = 64$, $h = 10^{-3}$, 2200 iterations, AM 1-step), we calculate the errors between $\boldsymbol{x}_{\mathcal{NN}}$ and the original function over the interval [0,15]. As shown in \autoref{fig:Glycolytic_Error}, the learned system accurately captures the dynamics' form. We observe that for high-dimensional cases, our error remains small over an extended period, indicating the stability and reliability of our approach.
\begin{figure}[t!]
  \centering
  \includegraphics[width=1\linewidth]{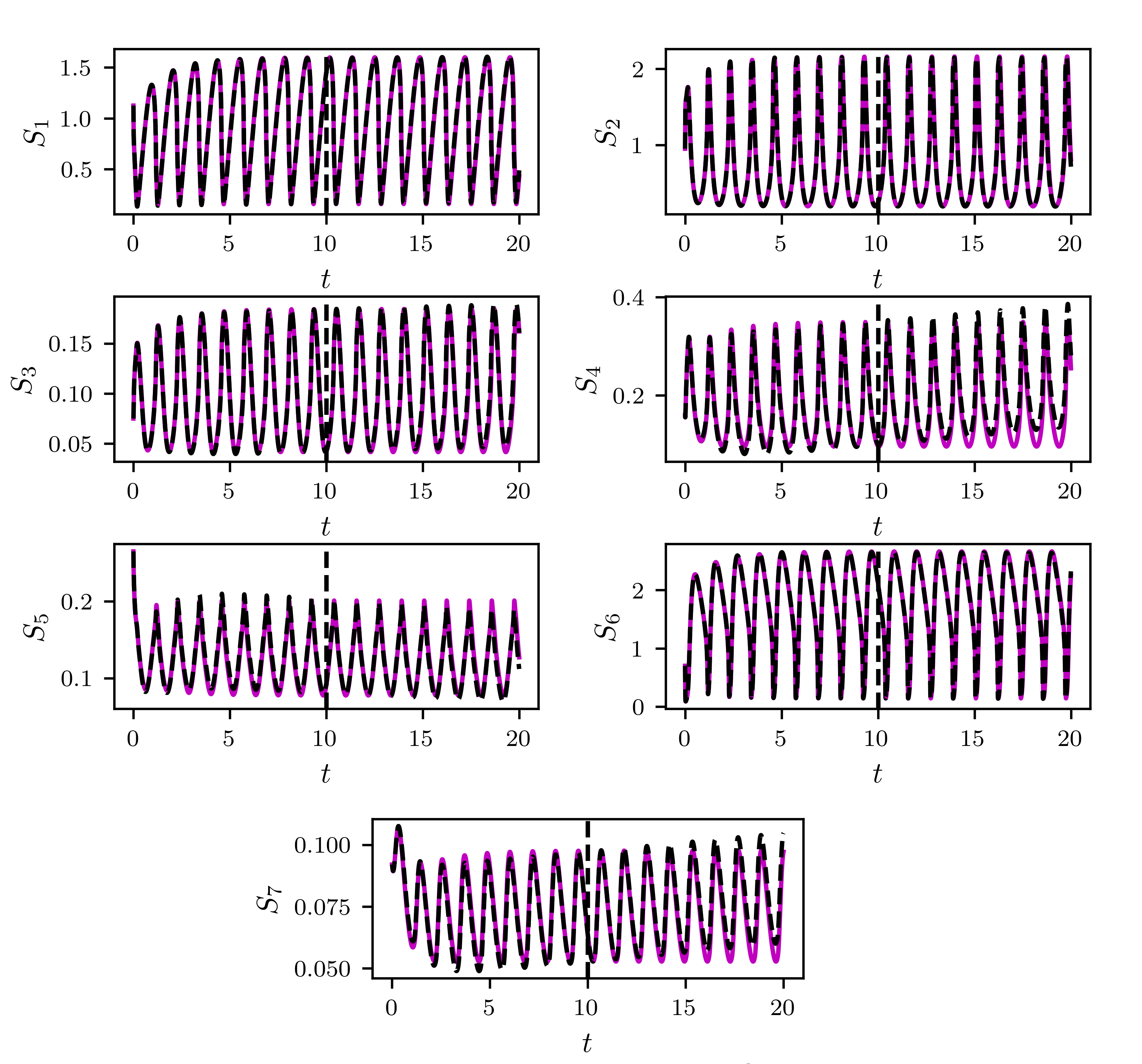}
  \caption{Comparison of trajectories for the Gylcolytic Oscillator: exact vs. learned dynamics. Solid blue lines indicate the exact dynamics, whereas dashed black lines represent the learned dynamics. The black dashed vertical line at $t=10$ demarcates the boundary between the `training' interval $[0,10]$ and the `prediction' interval $[10,20]$.}
  \label{fig:Glycolytic_Error}
\end{figure}

\subsection{Opinion dynamics}
This numerical experiment delves into the opinion dynamics model, a typical example within the realm of self-organized dynamical systems. The model is governed by the following system of equations:
\begin{equation}\label{eq:Opinion dynamics}
  \frac{d}{dt}\mathbf{x}_{i} = \alpha \sum_{j \neq i}a_{ij}(\mathbf{x}_{j}-\mathbf{x}_{i}), ~~ a_{ij}=\frac{\phi_{ij}}{\sum_{k}\phi_{ik}}, ~~ \phi_{ij}:=\phi(|\mathbf{x}_{j}-\mathbf{x}_{i}|).
\end{equation}

\begin{figure}[htp!]
\centering
\includegraphics[width=1\linewidth]{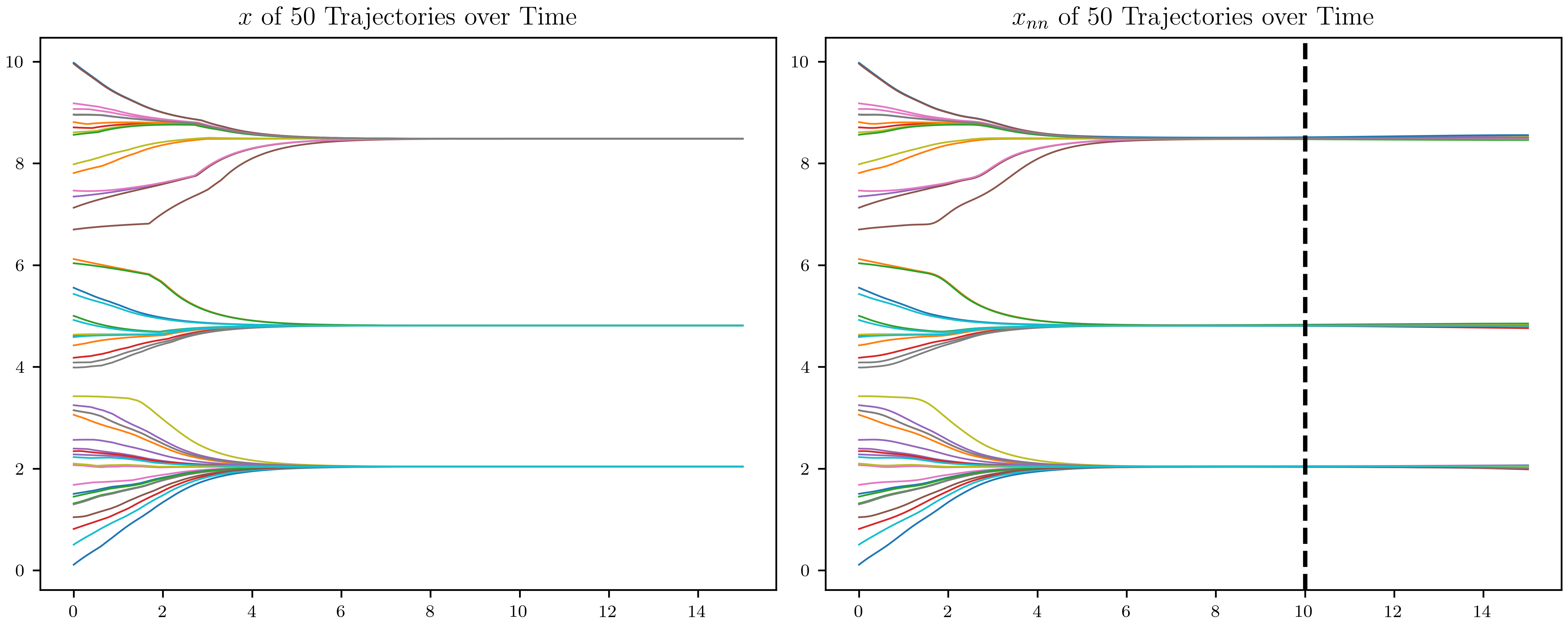}
\caption{A comparison of the temporal evolution of opinion dynamics as described by \eqref{eq:Opinion dynamics}, solved using a KAN, with the actual evolution. (Left figure: Actual evolution; Right figure: Network solutions.) The black dashed vertical line at $t=10$ demarcates the boundary between the `training' interval $[0,10]$ and the `prediction' interval $[10,15]$. }
\label{fig:Opinion dynamics_Slove_KAN}
\end{figure}

\begin{figure}[htp!]
\centering
\includegraphics[width=0.8\linewidth]{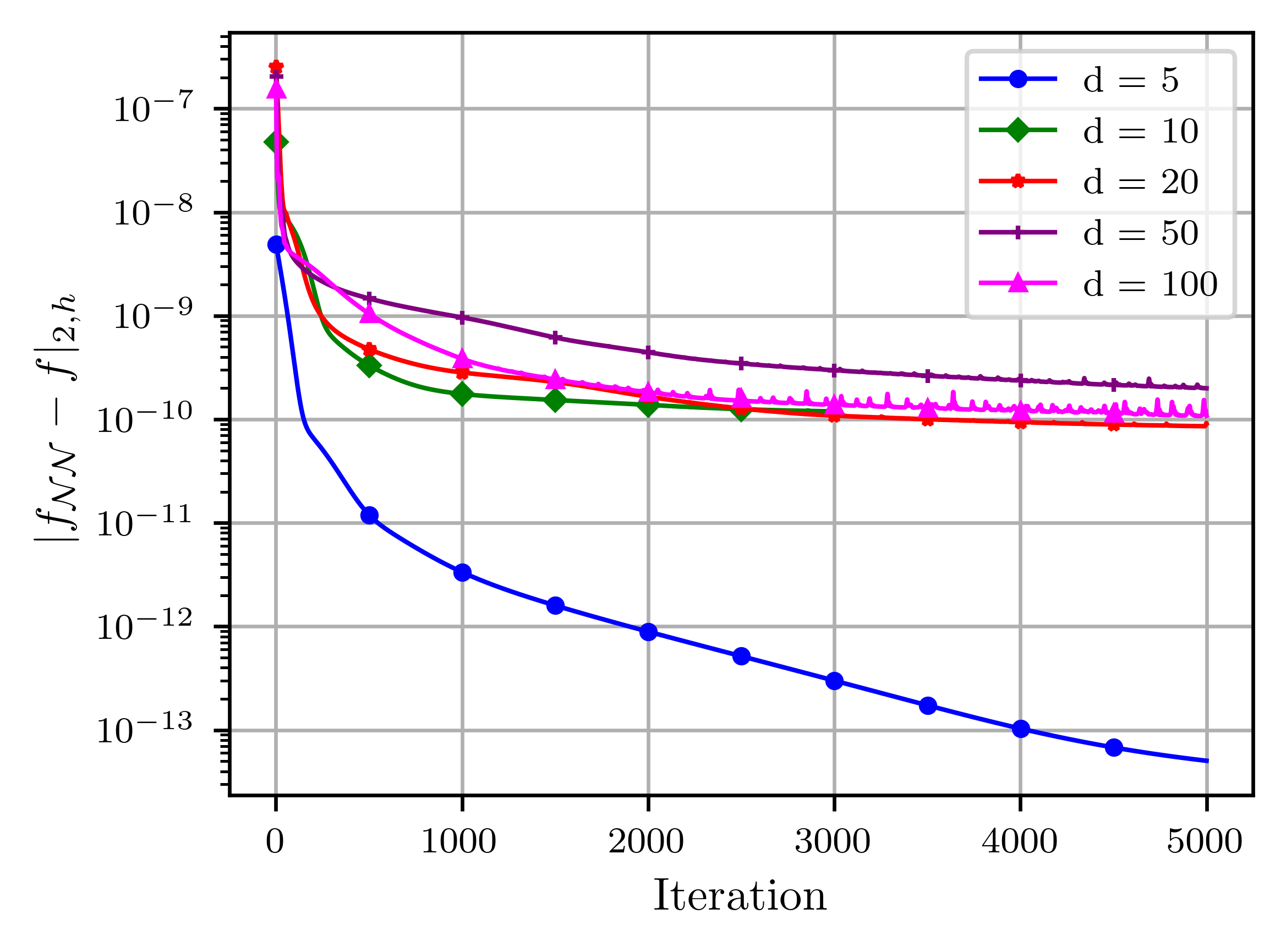}
\caption{Training Error between $\boldsymbol{f}$ and $\boldsymbol{f}_{\mathcal{NN}}$ of Opinion Dynamics Model at $T=10$.}
\label{fig:Opinion dynamics_Error_KAN}
\end{figure}

Here, \(\alpha > 0\) is a scaling parameter, and \(\phi\) is a scaled influence function defined as the characteristic function \(\phi=1\chi_{[0,1]}\), which operates on the "difference of opinions," denoted by \(|\mathbf{x}_{j}-\mathbf{x}_{i}|\). The metric \(|\cdot|\) signifies the absolute value.

We examine the system's behavior under random initial conditions across dimensions \(d = 50, 100, 200,\) and \(400\). The parameters for the KAN are set to \(k = 3\), \(G = 64\), and \(h = 10^{-3}\), with 3000 iterations per simulation. The Adams-Moulton method of order 1 is employed in conjunction with the KAN setup, with parameter arrays customized for different dimensions: [50 101 50], [100 201 100], [200 401 200], and [400 801 400].

\begin{table}[htp!]
\small
\centering
\renewcommand\arraystretch{1.2}
\caption{The $L^\infty$-norm error between the true state $\boldsymbol{x}(t)$ and the approximation  $\boldsymbol{x}_{\mathcal{NN}}(t)$ across the interval $[0,10].$ }
\label{tab:f_fnn_I}
\begin{tabular}{c|c|c|c|c}
\Xhline{2\arrayrulewidth}
$d$   &  50 & 100 & 200 & 400  \\
\hline
Errors  &  3.92960581e-02   & 1.62031044e-01  & 1.77045530e-01 & 1.71368496e-01 \\
\Xhline{2\arrayrulewidth}
\end{tabular}
\end{table}

A pivotal aspect of this study is the interplay between the KANs' approximation capability and the system's dimensionality \(d\). Our theoretical analysis indicates that the lower bound of the error, as quantified by the VC dimension, escalates with \(d\) but ultimately converges to a constant \(C\). This convergence suggests that the KAN model's approximation error is contained even amidst escalating system dimensionality, thereby ensuring the model's efficacy in high-dimensional scenarios. This attribute is indispensable for applications involving high-dimensional data, as it ascertains that the KAN model sustains a consistent approximation error margin. To substantiate these theoretical insights, we conduct a comprehensive error analysis for dimensions \(d = 50, 100, 200,\) and \(400\), as detailed in \autoref{tab:f_fnn_I}. \autoref{fig:Opinion dynamics_Slove_KAN} visually corroborates the algorithm's proficiency in approximating the 50-dimensional system over the interval \([0, 10]\) by contrasting it with the original orbit plot.

The graph in \autoref{fig:Opinion dynamics_Error_KAN} illustrates the convergence behavior of the error bounds for different dimensions
$d$ in the context of KANs. The error's lower bound is derived from the VC dimension. Conversely, the upper bound reflects the network's ability to approximate the target function.

For lower dimensions ($d=5$), the graph shows a rapid decrease in error as the network undergoes optimization, suggesting a strong approximation capability. As the dimension $d$ increases, the error's lower bound eventually approaching a fixed upper limit. This behavior is observed in the graph, where the error for higher dimensions ($d=10,20,50,100$) stabilizes at a nearly constant value. 

\bibliographystyle{abbrv}

\phantomsection
\addcontentsline{toc}{section}{\refname}
\bibliography{reference}

\end{document}